\documentclass[12pt,a4paper]{article}
\usepackage[a4paper,margin=1in,portrait]{geometry}
\usepackage{amssymb}
\usepackage{amsmath}
\usepackage{amsthm}
\usepackage[T1]{fontenc}
\usepackage[utf8]{inputenc}
\usepackage{times}
\usepackage{mathabx} 
\usepackage{changepage}
\usepackage{parcolumns}
\usepackage[all]{xy}
\usepackage{extarrows}
\usepackage{mathrsfs}
\usepackage[usenames,dvipsnames]{xcolor}
\usepackage{enumitem}
\usepackage{cancel}
\usepackage{hyperref}
\usepackage{memhfixc}
\usepackage{mathtools}
\usepackage{mdframed}

\setlist[enumerate]{topsep=0pt,itemsep=-1ex,partopsep=1ex,parsep=1ex}

\DeclareMathAlphabet\mathbb{U}{msb}{m}{n}
\usepackage{eucal}
\newtheorem{theorem}{Theorem}[section]
\newtheorem{corollary}[theorem]{Corollary}
\newtheorem{proposition}[theorem]{Proposition}
\newtheorem{lemma}[theorem]{Lemma}

\newcommand{\detail}[1]{}
\theoremstyle{definition}
\newtheorem{definition}[theorem]{Definition}
\theoremstyle{remark}
\newtheorem{remark}[theorem]{Remark}

\newcommand{\Gm}[2]{{\cE^{#2}_{\cM}(#1)}}
\newcommand{\Gn}[2]{{\cN^{#2}(#1)}}

\newcommand{\GGm}[3]{{\cE^{#3}_{\cM}[#1,#2]}}

\newcommand{\GMFq}[2]{{\cG^\Delta[#1,#2]}}
\newcommand{\GMFb}[2]{{\cE^\Delta[#1,#2]}}
\newcommand{\GMFm}[2]{{\cE^\Delta_{\cM}[#1,#2]}}

\newcommand{\GHOMq}[2]{{\Hom_\cG^{\Delta}[#1,#2]}}
\newcommand{\GHOMb}[2]{{\Hom_\cE^{\Delta}[#1,#2]}}
\newcommand{\GHOMm}[2]{{\Hom_\cM^{\Delta}[#1,#2]}}

\newcommand{\GHYBq}[2]{{\cG^{\Delta,h}[#1,#2]}}
\newcommand{\GHYBb}[2]{{\cE^{\Delta,h}[#1,#2]}}
\newcommand{\GHYBm}[2]{{\cE^{\Delta,h}_{\cM}[#1,#2]}}

\newcommand{\fS}{\mathfrak{S}}

\newcommand{\Gvb}[2]{{{#1}^#2}}
\newcommand{\Gvm}[2]{{{#1}_\cM^#2}}
\newcommand{\Gvn}[2]{{{#1}_\cN^#2}}
\newcommand{\Gvq}[2]{{\widetilde{#1}^#2}}

\newcommand{\GPb}[1]{{#1^\Delta}}
\newcommand{\GPc}[1]{{#1^\Delta_c}}
\newcommand{\GPm}[1]{{#1^\Delta_{c,\cM}}}
\newcommand{\GPq}[1]{{\widetilde{#1}^\Delta_c}}
\newcommand{\GVBb}[1]{{#1^{\Delta}}}
\newcommand{\GVBm}[1]{{#1^{\Delta}_{vb,\cM}}}
\newcommand{\GVBq}[1]{{\widetilde{#1}^{\Delta}_{vb}}}

\newcommand{\Hom}{{\operatorname{Hom}}}

\newcommand{\csn}[1]{{{\operatorname{csn}}}(#1)} 
\newcommand{\ud}{\mathrm{d}} 
\newcommand{\fX}{\mathfrak{X}} 
\newcommand{\cD}{\mathcal{D}}

\newcommand{\TOv}[1]{{ S(#1)}} 
\newcommand{\TOvz}[1]{{ S^0(#1)}} 

\newcommand{\pd}{\partial}
\newcommand{\e}{\varepsilon}

\newcommand{\bC}{\mathbb{C}}
\newcommand{\bE}{\mathbb{E}}
\newcommand{\bF}{\mathbb{F}}

\newcommand{\bK}{\mathbb{K}}
\newcommand{\bN}{\mathbb{N}}
\newcommand{\bR}{\mathbb{R}}

\newcommand{\Riem}{\mathrm{Riem}}

\newcommand{\cA}{\mathcal{A}}
\newcommand{\cB}{\mathcal{B}}

\newcommand{\cE}{\mathcal{E}}

\newcommand{\cG}{\mathcal{G}}

\newcommand{\cM}{\mathcal{M}}
\newcommand{\cN}{\mathcal{N}}

\newcommand{\csub}{\subset \subset}
\newcommand{\coleq}{\mathrel{\mathop:}=}
\newcommand{\eqcol}{=\mathrel{\mathop:}}

\DeclareMathOperator{\id}{id}

\DeclareMathOperator{\supp}{supp}

\DeclareMathOperator{\pr}{pr}

\newcommand{\tang}{\mathrm{T}}

\newcommand{\Lin}{\mathrm{L}}

\providecommand{\norm}[1]{\left\lVert#1\right\rVert}
\providecommand{\abso}[1]{\left\lvert#1\right\rvert}

  \makeatletter
    \let\OLD@otimes\otimes
    \def\opotimes{\@ifnextchar_\otimesWSB\otimesNSB}
    \def\otimesWSB_#1{\@ifnextchar^{\otimesWSBWXP_{#1}}{\otimesWSBNXP_{#1}}}
    \def\otimesWSBNXP_#1{\mathbin{\operatorname*{\OLD@otimes}_{#1}}}
    \def\otimesWSBWXP_#1^#2{\mathbin{\operatorname*{\OLD@otimes}_{#1}^{#2}}}
    \def\otimesNSB{\@ifnextchar^\otimesWXP\OLD@otimes}
    \def\otimesWXP^#1{\@ifnextchar_{\otimesWXPWSB^{#1}}{\otimesWXPNSB^{#1}}}
    \def\otimesWXPNSB^#1{\mathbin{\operatorname*{\OLD@otimes}^{#1}}}
    \def\otimesWXPWSB^#1_#2{\mathbin{\operatorname*{\OLD@otimes}^{#1}_{#2}}}
  \makeatother

\newcommand{\cL}{{\mathscr{L}}}

\newcommand{\rSO}{\rho^{SO}} 

\newcommand{\csna}{\mathfrak{p}}

\newcommand{\SK}[1]{{\mathrm{SK}(#1)}}

\newcommand{\SO}[1]{{\mathrm{SO}(#1)}}

\newcommand{\VSO}[1]{{\mathrm{VSO}(#1)}}

\def\mytitle{Nonlinear generalized sections and\\
vector bundle homomorphisms in\\
Colombeau spaces of generalized functions}

\begin{document}
\author{E.~A.~Nigsch}
\title{\mytitle}

\maketitle

\begin{abstract}
 We define and characterize spaces of manifold-valued generalized functions and generalized vector bundle homomorphisms in the setting of the full diffeomorphism-invariant vector-valued Colombeau algebra. Furthermore, we establish point value characterizations for these spaces.
\end{abstract}

\section{Introduction}

Lately, the theory of nonlinear generalized functions in the sense of Colombeau (\cite{Biagioni,colmult, GKOS, zbMATH01226424, MOBook}) is experiencing a fundamental restructuring of its geometrical variant. After the very early split between so-called full and special (or simplified) algebras, the special algebras -- due to their simpler structure -- very quickly were extended to geometrical contexts (\cite{MKnonsmooth,KSVintrins,Kgenmf,GKOS, GroupAnalysis, genrel}), with applications in particular in general relativity or Lie group analysis of differential equations.

The inherent limitations of special Colombeau algebras (i.e., one can have neither diffeomorphism invariance nor a canonical embedding of distributions) always fueled the desire for a corresponding development in the full branch of Colombeau theory. There, the geometrization of the theory went from $\cG^d$ (the diffeomorphism invariant algebra on open subsets of Euclidean space, \cite{found}) over $\hat\cG$ (its global version on manifolds, \cite{global}) to $\hat\cG^r_s$ (the tensorial version, \cite{global2}). At this point development stalled, essentially due to technical ballast which was dragged along since the very inception of the theory. The impossibility to obtain the sheaf property and a sensible notion of covariant derivative in the context of $\hat\cG^r_s$ made it necessary to reconsider the foundations of the theory -- this led to the development of the functional analytic approach in \cite{papernew}, incorporating earlier ideas of \cite{1177.46033,zbMATH06162234}.

This approach made it possible to obtain a global theory of spaces $\cG^\Delta$ of nonlinear generalized functions and sections of vector bundles, including a sensible notion of generalized covariant derivative. It is noteworthy that at the same time this did not inflict more technicalities upon the theory but, quite to the contrary, many constructions became conceptually easier.

However, the price one has to pay for ever wider generalizations of the theory is that one has to lift previously obtained results to the new context. The aim of this article is to establish the following constructions from the special Colombeau algebra (cf.~\cite{Kgenmf,genpseudo}) and results from \cite{mfval} in the setting of the full algebra $\cG^\Delta$:
\begin{enumerate}[label=(\roman*)]
 \item the space $\cG[X,Y]$ of manifold-valued generalized mappings (Section \ref{sec4}),
 \item the space $\Hom_\cG[E,F]$ of generalized vector bundle homomorphisms (Section \ref{sec5}),
 \item the space $\cG^h[M,F]$ of hybrid generalized mappings (Section \ref{sec6}),
 \item point value characterizations in the spaces $\cG^\Delta(E)$, $\cG[X,Y]$, $\Hom_\cG[E,F]$ and $\cG^h[M, F]$ (Section \ref{sec9}).
\end{enumerate}

Note that although already \cite{mfval} was dealing with manifold-valued generalized functions in full Colombeau spaces, this was in the (too restrictive) setting of $\hat\cG$, where not even an intrinsic point value characterization is possible.

\section{Preliminaries and Notation}

We assume all manifolds to be Hausdorff, second-countable and finite dimensional. $\cA_M$ denotes an arbitrary but fixed atlas of the manifold $M$. For a vector bundle $E$ over $M$, $\pi_E \colon E \to M$ is the projection onto the base manifold and $\fS_E$ denotes an arbitrary but fixed trivializing covering of $E$. $E_x$ denotes the fiber of $E$ over $x \in M$, $\Hom(E,F)$ the space of vector bundle homomorphisms from $E$ to another vector bundle $F$ and $\Gamma(E)$ the space of smooth sections of $E$. $\cD'(M,E)$ is the space of $E$-valued distributions on $M$ (\cite[Definition 3.1.4, p.~231]{GKOS}).

The field $\bK$ is fixed to be either $\bR$ or $\bC$. $A \csub B$ means that $A$ is a compact subset of the interior $B^\circ$ of $B$, we set $I \coleq (0,1]$. $\Riem(M)$ denotes the set of Riemannian metrics on the manifold $M$; for $h \in \Riem(M)$, $d_h$ denotes the corresponding Riemannian distance on $M$ and $B^h_r(p)$ the metric ball of radius $r$ with center $p\in M$ with respect to this distance.

Given any locally convex spaces $\bE_1, \dotsc, \bE_n$ and $\bF$ we denote by $\Lin(\bE_1, \dotsc, \bE_n; \bF)$ the space of all bounded $n$-linear mappings from $\bE_1 \times \dotsc \times \bE_n$ into $\bF$ and by $\cL(\bE_1, \dotsc, \bE_n; \bF)$ the space of continuous such mappings. The set of continuous seminorms of $\bE$ is denoted by $\csn{\bE}$.

$C^\infty(M)$ denotes the space of $\bK$-valued smooth functions on $M$ and $C^\infty_0(M)$ its subspace formed by functions with compact support.

As is customary for diffeomorphism invariant algebras of generalized functions, for a notion of smooth functions between infinite dimensional locally convex spaces we employ the setting of \emph{convenient calculus} (\cite{KM,Froe}). Among its features is the exponential law $C^\infty(\bE_1 \times \bE_2, \bF) \cong C^\infty(\bE_1, C^\infty(\bE_2, \bF))$; notationally, this is expressed as $f(x,y) = f^\vee(x)(y)$ and $g^\wedge(x,y) = g(x)(y)$.

Let $\bE$ be a locally convex space. We recall from \cite[27.2, p.~264]{KM} that given a convenient locally convex space $\bE$ and manifolds $M,N$, a mapping $f \colon \bE \times M \to N$ is called \emph{smooth} if for each $(\lambda, p) \in \bE \times M$ and each chart $(V, \psi)$ on $N$ with $f(\lambda,p) \in V$ there exists a $c^\infty$-open subset $W \subseteq \bE$ and a chart $(U, \varphi)$ on $M$ such that $(\lambda, p) \in W \times U$, $f(W \times U) \subseteq V$ and $\psi \circ f \circ (\id \times \varphi^{-1}) \in C^\infty(W \times \varphi(U), \psi(V))$. The space of all such mappings is denoted by $C^\infty(\bE \times M, N)$.

We will constantly work with differentials of local representatives of such maps as follows: given $f \in C^\infty(\bE \times M, N)$, charts $(U,\varphi)$ of $M$ and $(V, \psi)$ of $N$ as well as a pair $(\lambda, p) \in \bE \times U$ such that $f(\lambda, p ) \in V$, there exists an open subset $U_0 \subseteq U$ and a $c^\infty$-open neighborhood $W$ of $\lambda$ such that $f(W \times U_0) \subseteq V$ and
\[ \psi \circ f \circ (\id \times \varphi^{-1}) \in C^\infty(W \times \varphi(U_0), \psi(V)). \]
This mapping has partial differentials (\cite[Section 2.2.1, p.~111]{GKOS})
\begin{align*}
\ud_1^k \ud_2^l (\psi \circ f \circ (\id \times \varphi^{-1})) & \in C^\infty(W \times \varphi(U_0), \Lin(\bE, \dotsc, \bE, \bR^m, \dotsc, \bR^m; \bR^n)) \\
& \cong C^\infty(W \times \varphi(U_0), \Lin ( \bE, \dotsc, \bE; \Lin(\bR^m, \dotsc, \bR^m; \bR^n)))
\end{align*}
where $\dim M = m$, $\dim N = n$, $\bE$ appears $k$ times, $\bR^m$ $l$ times, and the isomorphism is given by the exponential law \cite[5.2, p.~53]{KM}. Hence, for any $(\lambda,p) \in W \times U_0$ and $\lambda_1, \dotsc, \lambda_k \in \bE$ we have
\[ \ud_1^k \ud_2^l ( \psi \circ f \circ (\id \times \varphi^{-1}))(\lambda, \varphi(p)) ( \lambda_1, \dotsc, \lambda_k ) \in \Lin ( \bR^m, \dotsc, \bR^m ; \bR^n). \]

Finally, we recall the following lemmas from \cite[Lemma 3.2.5, p.~281]{GKOS} and \cite[Lemma 2.5, p.~36]{Kgenmf}:

\begin{lemma}\label{riemequiv}
 Let $h_i$ ($i=1,2$) be Riemannian metrics inducing the Riemannian distances $d_i$ on $M$. Then for $K,K' \csub M$ there exists $C>0$ such that $d_2(p,q) \le C d_1(p,q)$ for all $p \in K$, $q \in K'$.
\end{lemma}

\begin{lemma}\label{diffestimate}Let $\Omega \subseteq \bR^m$ and $\Omega' \subseteq \bR^n$ be open, $f \colon \Omega \to \Omega'$ continuously differentiable and suppose that $K \csub \Omega$. Then there exists $C>0$ such that $\norm{f(x) - f(y)} \le C \norm{x-y}$ for all $x,y \in K$. $C$ can be chosen as $C_1 \cdot \sup \{\, \norm{f(z)} + \norm { Df(z)} : z \in L \,\}$ where $L$ is any fixed compact neighborhood of $K$ in $\Omega$ and $C_1$ only depends on $L$.
\end{lemma}

Moreover, we will use the following result from Riemannian geometry which is easily obtained using normal coordinates:

\begin{lemma}\label{kartenschaetz}
Given any point $p_0$ of a manifold $M$ there exists a Riemannian metric $g$ on $M$ and a chart $(U, \varphi)$ containing $p_0$ such that $d_g(p,q) = \norm{ \varphi(p) - \varphi(q) }$ for all $p,q \in V$. 
\end{lemma}


\section{Nonlinear generalized sections: a review}\label{sec3}

For the convenience of the reader we recall the basic definitions of the algebra $\cG^\Delta(E)$ of nonlinear generalized sections of a vector bundle $E$ from \cite{bigone}.

\begin{definition}
 Let $\Delta$ be a (possibly empty) finite set of vector bundles over arbitrary manifolds and $E \to M$ a vector bundle. Then we set $\VSO{G} \coleq \cL(\cD'(N, G); \Gamma(G))$ for any element $G \to N$ of $\Delta$, $\VSO{\Delta} \coleq \prod_{G \in \Delta} \VSO{G}$, $\cE^\Delta(E) \coleq C^\infty(\VSO{\Delta}, \Gamma(E))$, $\cE^\Delta(M) \coleq \cE^\Delta(M \times \bK)$ and $\cE^\emptyset(E) = \Gamma(E)$.
\end{definition}

\begin{definition}A \emph{test object} on a vector bundle $E \to M$ is a net
 $(\Phi_\e)_\e \in \VSO{E}^I$ satisfying the following conditions.
\begin{enumerate}
 \item[(VSO1)] ($\forall g \in \Riem(M)$) ($\forall x_0 \in M$) ($\exists$ an open neighborhood $V$ of $x_0$) ($\forall r>0$) ($\exists \e_0 \in I$) ($\forall x \in V$) ($\forall \e \le \e_0$) ($\forall u \in \cD'(M, E)$): $(u|_{B^g_{r}(x)} = 0 \Rightarrow \Phi_\e(u)(x) = 0)$.
\item[(VSO2)] $\Phi_\e \to \id$ in $\cL(\cD'(M, E); \cD'(M, E))$.
\item[(VSO3)] ($\forall \csna \in \csn{\VSO{E}}$) ($\exists N \in \bN$): $\csna(\Phi_\e) = O(\e^{-N})$.
 \item[(VSO4)] ($\forall \csna \in \csn{\cL(\Gamma(E); \Gamma(E))}$) ($\forall m \in \bN$): $\csna(\Phi_\e|_{\Gamma(E)} - \id) = O(\e^m)$.
\end{enumerate}
A \emph{0-test object} is a sequence $(\Phi_\e)_\e \in \VSO{E}^I$ satisfying (VSO1), (VSO3) and the following conditions.
\begin{enumerate}
\item[(VSO2')] $\Phi_\e \to 0$ in $\cL(\cD'(M, E); \cD'(M, E))$.
\item[(VSO4')] ($\forall \csna \in \csn{\cL(\Gamma(E); \Gamma(E))}$) ($\forall m \in \bN$): $\csna(\Phi_\e|_{\Gamma(E)}) = O(\e^m)$.
\end{enumerate}
We denote by $\TOv{E}$ the set of all test objects and by $\TOvz{E}$ the set of all $0$-test objects. The notations $\TOv{\Delta}$ and $\TOvz{\Delta}$ mean families of corresponding test objects indexed by $\Delta$, e.g.,
\[ \TOv{\Delta} \coleq \prod_{G \in \Delta} \TOv{G} = \{\, ( ( \Phi_{G,\e} )_\e )_G : (\Phi_{G,\e})_\e \in \TOv{G}\ \forall G \in \Delta\,\}. \]
\end{definition}

\begin{definition}\label{scalmodneg}
 $R \in \cE^\Delta(E)$ is called \emph{moderate} if ($\forall \csna \in \csn{\Gamma(E)}$) ($\forall k \in \bN_0$) ($\forall (\Phi_{\e})_\e \in \TOv{\Delta}$, $(\Psi_{1,\e})_\e$, $\dotsc$, $(\Psi_{k,\e})_\e \in \TOvz{\Delta}$) ($\exists N \in \bN$) such that we have
\[ \csna (( \ud^k R)(\Phi_\e)(\Psi_{1,\e},\dotsc,\Psi_{k,\e})) = O(\e^{-N}). \]
The space of moderate elements of $\cE^\Delta(E)$ is denoted by $\cE^\Delta_\cM(E)$.
Similarly, $R$ is called \emph{negligible} if and only if ($\forall \csna \in \csn{\Gamma(E)}$) ($\forall k \in \bN_0$) ($\forall (\Phi_{\e})_\e \in \TOv{\Delta}$, $(\Psi_{1,\e})_\e$, $\dotsc$, $(\Psi_{k,\e})_\e \in \TOvz{\Delta}$) ($\forall m \in \bN$) we have
\[ \csna (( \ud^k R)(\Phi_\e)(\Psi_{1,\e},\dotsc,\Psi_{k,\e})) = O(\e^m). \]
The space of negligible elements of $\cE^\Delta(E)$ is denoted by $\cN^\Delta(E)$.
\end{definition}

The following characterization (\cite[Theorem 6.3, p.~204]{bigone}) will be useful:

\begin{theorem}\label{nonegder}$R \in \Gm{E}{\Delta}$ is negligible if ($\forall K \csub M$) ($\forall m \in \bN$) ($\forall (\Phi_{\e})_{\e} \in \TOv{\Delta}$):
$\sup_{x \in K} \norm { R ( \Phi_{\e} ) } = O(\e^m)$, where $\norm{\cdot}$ is the norm on $\Gamma(E)$ induced by any Riemannian metric on $E$.
\end{theorem}

Note that, more generally than in \cite{bigone}, we do not require the elements of $\Delta$ to have $M$ as their base manifold.

\section{Manifold-valued generalized functions}\label{sec4}

In this section let $M,N,P$ be manifolds with atlases $\cA_M,\cA_N,\cA_P$, respectively. Our definitions and results will (as is easily verified in each case) not depend on the choice of the atlases.

\begin{definition}
$\GMFb{M}{N} \coleq C^\infty(\VSO{\Delta} \times M, N)$ is called the \emph{basic space} of Colombeau generalized functions on $M$ taking values in $N$ with index set $\Delta$.
\end{definition}

\begin{definition}$R \in \GMFb{M}{N}$ is called \emph{c-bounded} if
\[ (\forall K \csub M)\ (\exists L \csub N)\ (\forall (\Phi_\e)_\e \in S(\Delta))\ (\exists \e_0>0)\ (\forall \e < \e_0): R(\Phi_\e, K) \subseteq L. \]
\end{definition}

\begin{definition} Given $R \in \GMFb{M}{N}$ and $S \in \GMFb{N}{P}$ we define their \emph{composition} $S \circ R \in \GMFb{M}{P}$ as
\[ (S \circ R)(\Phi, x) \coleq S(\Phi, R(\Phi, x))\qquad (\Phi \in \VSO{\Delta}, x \in M). \]
\end{definition}

Given $R \in \GMFb{M}{N}$ and $f \in C^\infty(N)$ we see that $f \circ R \in \GMFb{M}{\bK} \cong  \cE^{\Delta}(M)$, where the isomorphism (which also defines c-boundedness for elements of $\cE^\Delta(M)$) is given by \cite[27.17, p.~273]{KM}.
Suppressing this isomorphism notationally we simply write $f \circ R \in \cE^{\Delta}(M)$. We now come to the quotient construction. For the sake of conciseness, let us introduce some notation:

Given charts $(U,\varphi)$ on $M$ and $(V, \psi)$ on $N$ as well as $R \in \GMFb{M}{N}$, we set $R_{\psi\varphi} \coleq \psi \circ R \circ (\id \times \varphi^{-1})$ where it is defined. Furthermore, given $\Phi_\e \in \VSO{\Delta}$ we set $R_\e \coleq R(\Phi_\e, . ) \in C^\infty(M,N)$.

\begin{definition}\label{manval_mod}An element $R \in \GMFb{M}{N}$ is called \emph{moderate} if 
 \begin{enumerate}[label=(\roman*)]
  \item\label{manval_mod.1} $R$ is c-bounded,
  \item\label{manval_mod.2} $(\forall (U, \varphi) \in \cA_M)$ $(\forall L \csub U)$ $(\forall (V, \psi) \in \cA_N)$ $(\forall L' \csub V)$ $(\forall k,l \in \bN_0)$ $(\forall (\Phi_\e)_\e \in S(\Delta)$, $(\Phi_{1,\e})_\e$, $\dotsc$, $(\Phi_{k,\e})_\e \in \TOvz{\Delta})$ $(\exists N \in \bN)$:
  \begin{equation}\label{manval_mod_eq}
\norm{ \ud_1^k \ud_2^l R_{\psi\varphi} (\Phi_\e, \varphi(p))(\Psi_{1,\e},\dotsc,\Psi_{k,\e})} = O(\e^{-N})
  \end{equation}
uniformly for $p \in L \cap R_\e^{-1}(L')$.
\end{enumerate}
By $\GGm{M}{N}{\Delta}$ we denote the space of all moderate elements of $\GMFb{M}{N}$.
\end{definition}
Note that $C^\infty(M,N) = \cE^\emptyset_{\cM} [M,N] \subseteq \GMFm{M}{N}$.

\begin{definition}\label{manval_equiv}Two elements $R,S \in \GMFm{M}{N}$ are called \emph{equivalent}, written $R \sim S$, if
\begin{enumerate}[label=(\roman*)]
 \item \label{manval_equiv.1} $(\forall L \csub M)$ $(\forall (\Phi_\e)_\e \in S(\Delta))$ $(\forall h \in \Riem(N))$:
 \begin{equation}\label{manval_equiv_eq}
 \sup_{p \in L} d_h(R_\e(p), S_\e(p)) \to 0,
 \end{equation}
 \item \label{manval_equiv.2}  $(\forall (U,\varphi) \in \cA_M)$ $(\forall L \csub U)$ $(\forall (V, \psi) \in \cA_N)$ $(\forall L' \csub V)$ $(\forall k,l \in \bN_0)$\\
$(\forall (\Phi_\e)_\e \in \TOv{\Delta}$, $(\Psi_{1,\e})_\e$, $\dotsc$, $(\Psi_{k,\e})_\e \in \TOvz{\Delta})$ $(\forall m \in \bN)$:
 \[
 \Bigl\lVert \ud_1^k \ud_2^l \Bigl(R_{\psi\varphi} -
  S_{\psi\varphi} \bigr)(\Phi_\e, \varphi(p))(\Psi_{1,\e},\dotsc,\Psi_{k,\e}) \bigr\rVert = O(\e^m)
 \]
uniformly for $p \in L \cap R_\e^{-1}(L') \cap S_\e^{-1}(L')$.
\end{enumerate}

We call $R$ and $S$ \emph{equivalent of order $0$}, written $R \sim_0 S$, if they satisfy \ref{manval_equiv.1} and \ref{manval_equiv.2} for $k=l=0$.
\end{definition}

\begin{remark}The following claims are straightforward to verify:
\begin{enumerate}[label=(\roman*)]
 \item Lemma \ref{kartenschaetz} show that Definition \ref{manval_equiv} \ref{manval_equiv.1} is equivalent to the same condition with $(\forall h \in \Riem(N))$ replaced by $(\exists h \in \Riem(N))$.
 \item Definition \ref{manval_equiv} defines equivalence relations $\sim$ and $\sim_0$ on $\GGm{M}{N}{\Delta}$.
 \item Because differentials are symmetric multilinear maps (\cite[5.11, p.~58]{KM}), one may assume by polarization that $(\Psi_{1,\e})_\e = \dotsc = (\Psi_{k,\e})_\e$ in Definition \ref{manval_mod} \ref{manval_mod.2} and Definition \ref{manval_equiv} \ref{manval_equiv.2}.
\end{enumerate}
\end{remark}

\begin{definition}The quotient space $\GMFq{M}{N} \coleq \GMFm{M}{N} / {\sim}$ is called the space of \emph{compactly bounded Colombeau generalized functions} from $M$ to $N$.
\end{definition}

There is a coordinate-free characterization of equivalence of order zero as follows: 

\begin{theorem}\label{charnegcoordfree}
For $R,S \in \GMFm{M}{N}$ the following statements are equivalent.
 \begin{enumerate}[label=(\alph*)]
  \item \label{charnegcoordfree.2} $R \sim_0 S$.
  \item \label{charnegcoordfree.3} $(\forall L \csub M)$ $(\forall m \in \bN)$ $(\forall (\Phi_\e)_\e \in S(\Delta))$ $(\forall h \in \Riem(N))$:
\[ \sup_{p \in L} d_h ( R_\e(p), S_\e(p) ) = O(\e^m). \]
\end{enumerate}
Moreover, \ref{charnegcoordfree.3} is equivalent to the same condition with $(\forall h \in \Riem(N))$ replaced by $(\exists h \in \Riem(N))$.
\end{theorem}
\begin{proof}
Suppose that Definition \ref{manval_equiv} \ref{manval_equiv.1} holds. Using Lemma \ref{kartenschaetz}, construct an atlas $\cB$ of $N$ such that for each $(V, \psi) \in \cB$ there is $h \in \Riem(N)$ satisfying $d_h ( p,q) = \norm{\psi(p) - \psi(q)}$ for $p,q \in V$. In \ref{charnegcoordfree.3} one can clearly assume that $L \csub U$ and $R_\e(L) \cup S_\e(L) \subseteq L' \csub V$ for charts $(U,\varphi) \in \cA_M$ and $(V, \psi) \in \cB$. But then the claim follows directly from the equality $d_h ( R_\e(p), S_\e(p)) = \norm { R_{\psi\varphi}(\Phi_\e, \varphi(p)) - S_{\psi\varphi}(\Phi_\e, \varphi(p))}$.
\end{proof}

We will now show that composition preserves moderateness and equivalence of order zero; in the sequel, this will be used to show that equivalence of order zero actually is the same as equivalence for elements of $\GMFm{M}{N}$.

\begin{proposition}\label{fctcomp}
Let $R \in \GMFm{M}{N}$ and $S \in \GMFm{N}{P}$. Then
\begin{enumerate}[label=(\roman*)]
 \item\label{fctcomp.1} $S \circ R \in \GMFm{M}{P}$,
 \item\label{fctcomp.2} if $R' \in \GMFm{M}{N}$ and $S' \in \GMFm{N}{P}$ are such that $R \sim_0 R'$ and $S \sim_0 S'$, then $S \circ R \sim_0 S' \circ R'$.
\end{enumerate}
\end{proposition}
\begin{proof}
\ref{fctcomp.1}: $S \circ R$ clearly is c-bounded. In order to verify its moderateness fix $k,l \in \bN_0$, charts $(U,\varphi) \in \cA_M$ and $(W,\rho) \in \cA_P$, $L \csub U$ and $L'' \csub W$. By c-boundedness of $R$
\[ (\exists L' \csub N)\ (\forall(\Phi_\e)_\e \in \VSO{\Delta})\ (\exists \e_0>0)\ (\forall \e<\e_0): R(\Phi_\e,L)\subseteq L'. \]
Cover $L'$ by finitely many charts $(V_j, \psi_j) \in \cA_N$ and write $L' = \bigcup_j L_j$ with $L_j \csub V_j$.
Then it suffices to show that
\begin{equation}\label{fctcomp_eq}
(\forall j)\ (\forall (\Phi_\e)_\e,(\Psi_\e)_\e)\ (\exists N_j): \\
\norm{ \ud_1^k \ud_2^l ( (S \circ R)_{\rho\varphi} )(\Phi_\e,\varphi(p))(\Psi_\e,\dotsc,\Psi_\e)} = O(\e^{-N_j})
\end{equation}
uniformly for $p \in L \cap R_\e^{-1}(L_j') \cap (S \circ R)_\e^{-1}(L'')$.
We write $L', V, \psi$ instead of $L_j', V_j, \psi_j$. Such a simplification will also be used (less verbosely) in subsequent proofs by saying: using c-boundedness of $R$, we can assume that $p \in L \cap R_\e^{-1}(L') \cap (S \circ R)_\e^{-1}(L'')$ where $L' \csub V$ for a chart $(V,\psi) \in \cA_N$.

Given any $\Phi_\e \in \VSO{\Delta}$ and $p \in L \cap R_\e^{-1}(L') \cap (S \circ R)_\e^{-1}(L'')$ we have (for small $\e$) the equality
\[
 (S \circ R)_{\rho\varphi} (\Phi_\e, \varphi(p)) = S_{\rho\psi} ( \Phi_\e, R_{\psi\varphi}(\Phi_\e, \varphi(p))).
\]
By the chain rule \cite[3.18, p.~33]{KM}, the differentials of this expression appearing in \eqref{fctcomp_eq} are given by terms of the form
\begin{multline*}
 \ud_1^k \ud_2^l S_{\rho\psi} ( \Phi_\e, R_{\psi\varphi}(\Phi_\e, \varphi(p))) ( \Psi_\e, \dotsc, \Psi_\e) \\
\cdot \ud_1^{i_1} \ud_2^{j_1} R_{\psi\varphi} ( \Phi_\e, \varphi(p )) ( \Psi_\e, \dotsc, \Psi_\e) \cdot \dotsc \cdot \ud_1^{i_l} \ud_2^{j_k} R_{\psi\varphi} ( \Phi_\e, \varphi(p )) ( \Psi_\e, \dotsc, \Psi_\e)
\end{multline*}
which are moderate by assumption.

\ref{fctcomp.2} First, suppose that $S \sim_0 S'$. Given $L \csub M$ choose $L' \csub N$ such that $R(\Phi_\e, L) \subseteq L'$ for small $\e$. Then for any $h \in \Riem(N)$, $S \circ R \sim_0 S' \circ R$ follows using Theorem \ref{charnegcoordfree} from
\[ \sup_{p \in L} d_h ((S \circ R)(\Phi_\e, p), (S' \circ R)(\Phi_\e,p)) \le \sup_{q \in L'} d_h ( S(\Phi_\e, q), S' (\Phi_\e, q)) = O(\e^m). \]

Next, we establish that $R \sim R'$ implies $S \circ R \sim S \circ R'$; this will show \ref{fctcomp.2} by transitivity of $\sim_0$. Fix $(U, \varphi) \in \cA_M$, $L \csub U$, $(W, \rho) \in \cA_P$ and $L'' \csub W$. By c-boundedness of $R$ and $R'$ it suffices to obtain the estimate in Definition \ref{manval_equiv} \ref{manval_equiv.2} for $p \in L \cap R_\e^{-1}(L') \cap R_\e^{\prime -1}(L') \cap (S \circ R)_\e^{-1}(L'') \cap (S \circ R')_\e^{-1}(L'')$, where $L' \csub (V,\psi)$ for a chart $(V,\psi) \in \cA_N$. In this case we can write $(S \circ R)_{\rho\varphi} ( \Phi_\e, \varphi(p)) - (S \circ R')_{\rho\varphi} (\Phi_\e, \varphi(p)))$ as
\[ S_{\rho\psi}(\Phi_\e, R_{\psi\varphi}(\Phi_\e, \varphi(p))) - S_{\rho\psi}(\Phi_\e, R'_{\psi\varphi}(\Phi_\e, \varphi(p))) \]
which by Lemma \ref{diffestimate} is $O(\e^m)$ as desired.
\end{proof}

We will now show that moderateness and equivalence can be characterized by composition with smooth functions. This idea goes back to \cite{KSVintrins} and is the key to establishing that equivalence of order zero is the same as equivalence, which is the analogoue of Theorem \ref{nonegder} for manifold-valued generalized functions. 

\begin{theorem}\label{charmod}
 Let $R \in \GMFb{M}{N}$ be c-bounded. Then the following statements are equivalent:
 \begin{enumerate}[label=(\alph*)]
  \item \label{charmod.1} $R \in \cE^\Delta_{\cM}[M,N]$.
  \item \label{charmod.4} For all manifolds $P$ and all $S \in \GGm{N}{P}{\Delta}$: $S \circ R \in \GGm{M}{P}{\Delta}$.
  \item \label{charmod.3} $\forall f \in C^\infty(N, \bR)$: $f \circ R \in \Gm{M}{\Delta}$.
  \item \label{charmod.2} $\forall f \in C^\infty_0(N, \bR)$: $f \circ R \in \Gm{M}{\Delta}$.
 \end{enumerate}
\end{theorem}

\begin{proof}
\ref{charmod.1} $\Rightarrow$ \ref{charmod.4} is Proposition \ref{fctcomp} \ref{fctcomp.1} and \ref{charmod.4} $\Rightarrow$ \ref{charmod.3} $\Rightarrow$ \ref{charmod.2} are trivial.
 
For \ref{charmod.2} $\Rightarrow$ \ref{charmod.1} fix $k,l \in \bN_0$, $(U, \varphi) \in \cA_M$, $L \csub U$, $(V, \psi) \in \cA_N$ and $L' \csub V$.
Choose $f \in C^\infty_0(N, \bR)^n$ such that $f \equiv \psi$ in an open neighborhood of $L'$. It then suffices to show the estimate \eqref{manval_mod_eq} with the chart $\psi$ replaced by $f$ and the estimate taken over $p \in L$ instead of $p \in L \cap R_\e^{-1}(L')$.

Applying the exponential law $C^\infty(\VSO{\Delta} \times \varphi(U)) \cong C^\infty(\VSO{\Delta}, C^\infty(\varphi(U)))$
we see that $\ud_1^k \ud_2^l ( f \circ R \circ (\id \times \varphi^{-1}))(\Phi_\e, \varphi(p))(\Psi_{1,\e}, \dotsc, \Psi_{k,\e})$ is given by
\[ \ud^l ( \ud^k ( ( f \circ R)^\vee )(\Phi_\e)(\Psi_{1,\e}, \dotsc, \Psi_{k,\e}) \circ \varphi^{-1} ). \]

Because the norm of the $l$th differential of any smooth function $\bR^m \to \bR^n$ can be estimated by partial derivatives of order $\le l$ of its component functions it suffices to obtain estimates for
\begin{equation}\label{hellerstern}
\sup_{p \in L} \norm { \pd^\alpha ( \ud^k ( (f^j \circ R)^\vee)(\Phi_\e)(\Psi_{1,\e},\dotsc,\Psi_{k,\e}) \circ \varphi^{-1} ) (\varphi(p))}
\end{equation}
where $\alpha \in \bN_0^n$ and $f^j$ denotes the $j$th component of $f$.

Because
\[ \csna \colon g \mapsto \sup_{x \in L} \abso{\pd^\alpha ( g \circ \varphi^{-1})(\varphi(x))} \]
is a continuous seminorm of $C^\infty(M)$, \eqref{hellerstern} is nothing else than \[ \csna ( \ud^k ( ( f^j \circ R)^\vee )(\Phi_\e)(\Psi_{1,\e}, \dotsc, \Psi_{k,\e})) \] which is moderate according to Definition \ref{scalmodneg} by assumption.
\end{proof}

A similar result holds for equivalence: 

\begin{theorem}\label{charneg}
For $R,S \in \GMFm{M}{N}$ the following statements are equivalent:
 \begin{enumerate}[label=(\alph*)]
  \item \label{charneg.1} $R \sim S$.
  \item \label{charneg.2} $R \sim_0 S$.
  \item \label{charneg.5} $f \circ R - f \circ S \in \Gn{M}{\Delta}$ for all $f \in C^\infty(N, \bR)$.
\item \label{charneg.4} $f \circ R - f \circ S \in \Gn{M}{\Delta}$ for all $f \in C^\infty_0(N, \bR)$.
\end{enumerate}
\end{theorem}
\begin{proof}
 \ref{charneg.1} $\Rightarrow$ \ref{charneg.2} is trivial, \ref{charneg.2} $\Rightarrow$ \ref{charneg.5} is Proposition \ref{fctcomp} \ref{fctcomp.2} and \ref{charneg.5} $\Rightarrow$ \ref{charneg.4} is trivial.
 \ref{charneg.4} $\Rightarrow$ \ref{charneg.1}: assume that Definition \ref{manval_equiv} \ref{manval_equiv.1} does not hold, i.e., $(\exists \delta>0)$ $(\exists L)$ $(\exists (\Phi_\e)_\e)$ $(\exists h)$ $(\exists (\e_k)_k \searrow 0)$ $(\exists (p_k)_k \in L^\bN)$: $d_h(R(\Phi_{\e_k}, p_k), S(\Phi_{\e_k}, p_k)) > \delta$ for all $k$. Because of c-boundedness of $R$ and $S$, taking subsequences we can assume that $R(\Phi_{\e_k}, p_k) \to x$ and $S(\Phi_{\e_k}, p_k) \to y \ne x$. Choose $f \in C^\infty_0(N, \bR)$ such that $f(x) \ne f(y)$. Then \ref{charneg.4} implies
\[ \abso{ f ( R(\Phi_{\e_k}, p_k)) - f ( S(\Phi_{\e_k}, p_k)) } \to 0 \]
which gives a contradiction. Hence, Definition \ref{manval_equiv} \ref{manval_equiv.1} must hold.

For Definition \ref{manval_equiv} \ref{manval_equiv.2}, given a chart $(V, \psi)$ of $N$ used in the test and $L' \csub V$, choose $f \in C^\infty_0(N, \bR)^n$ such that $f \equiv \psi$ in an open neighborhood $V' \subseteq V$ of $L'$. Then we can replace $\psi$ by $f$ for the estimate and instead of $p \in L \cap R(\Phi_\e,.)^{-1}(L') \cap S(\Phi_\e,.)^{-1}(L')$ estimate over $p \in L$, from which the claim is obvious as in Theorem \ref{charmod} \ref{charmod.2} $\Rightarrow$ \ref{charmod.1}.
\end{proof}

Hence, composition is well-defined as a map $\GMFq{M}{N} \times \GMFq{N}{P} \to \GMFq{M}{P}$.

\section{Generalized vector bundle homomorphisms}\label{sec5}

Let $E \to M$ and $F \to N$ be vector bundles with typical fibers $\bE$ and $\bF$ (which are finite dimensional vector spaces over $\bK$), dimensions $m'$ and $n'$ and trivializing coverings $\fS_E$ and $\fS_F$, respectively. For any vector bundle homomorphism $s \in \Hom(E,F)$ we denote by $\underline{s} \in C^\infty(M, N)$ the underlying smooth map between the basic manifolds such that $\pi_F \circ s = \underline{s} \circ \pi_E$. By $\Hom_c(E, \bR \times \bR^{n'})$ we denote the set of all $f \in \Hom(E, \bR \times \bR^{n'})$ such that $\underline{f} \colon M \to \bR$ has compact support.

\begin{definition}\label{hombase}By $\GHOMb{E}{F}$ we denote the space of all mappings $f \in C^\infty(\VSO{\Delta} \times E, F)$ for which there exists a mapping $\underline{f} \in C^\infty(\VSO{\Delta} \times M, N) = \GMFb{M}{N}$ satisfying $\pi_F \circ f = \underline{f} \circ (\id \times \pi_E)$ and such that for all $\lambda \in \VSO{\Delta}$ and all $x \in E$ the restriction $f (\lambda, .) |_{E_x} \colon E_x \to F_{\underline{f}(x)}$ is linear. $\GHOMb{E}{F}$ is called the \emph{basic space} of generalized vector bundle homomorphisms from $E$ to $F$.
\end{definition}

We introduce the following notation. Let $f \in \GHOMb{E}{F}$, $(U, \tau) \in \fS_E$, $(V, \sigma) \in \fS_F$, $(U_0, \varphi) \in \cA_M$, $\lambda_0 \in \bE$ and $p_0 \in U \cap U_0 \cap \underline{f}(\lambda_0, .)^{-1}(V)$. Then there is a $c^\infty$-open neighborhood $W$ of $\lambda_0$ and an open subset $\widetilde U \subseteq U \cap U_0$ of $p_0$ such that $\underline{f} ( W \times \widetilde U) \subseteq V$ and $\sigma \circ f \circ (\id \times \tau^{-1}) \in C^\infty(W \times \widetilde U \times \bE, V \times \bF)$. We write the components of this function as in 
\[ (\sigma \circ f \circ (\id \times \tau^{-1}))(\lambda,p,v) = (\underline{f} ( \lambda, p), \widearrow f_{\sigma\tau} ( \lambda,p)\cdot v)\qquad (\lambda \in W, p \in \widetilde U, v \in \bE) \]
with $\widearrow f_{\sigma\tau} \in C^\infty(W \times \widetilde U, \Lin(\bE; \bF))$. We see that the map
\[ \widearrow f_{\sigma\tau} \circ (\id \times \varphi^{-1}) \in C^\infty(W \times \varphi(\widetilde U), \Lin(\bE; \bF)) \]
has derivatives
\[ \ud_1^k \ud_2^l ( \widearrow f_{\sigma \tau} \circ (\id \times \varphi^{-1})) ( \lambda, \varphi(p))(\lambda_1, \dotsc, \lambda_k) \in \Lin (\bR^m, \dotsc, \bR^m; \Lin(\bE; \bF)) \]
(with $\bR^m$ appearing $l$ times) existing for all $\lambda \in W$, $p \in \widetilde U$ and all $\lambda_1, \dotsc, \lambda_k \in \VSO{\Delta}$. Again, we set $R_\e \coleq R(\Phi_\e, .)$.

\begin{definition}\label{hommod}
 We say that $R \in \GHOMb{E}{F}$ is \emph{moderate} if
\begin{enumerate}[label=(\roman*)]
\item \label{hommod.1} $\underline{R} \in \GMFm{E}{F}$,
\item \label{hommod.2} $(\forall (U, \tau) \in \fS_E)$ $(\forall (U_0, \varphi) \in \cA_M)$ $(\forall (V, \sigma) \in \fS_F)$ $(\forall k,l \in \bN_0)$ $(\forall L \csub U_0 \cap U)$ $(\forall L' \csub V)$ $(\forall (\Phi_\e)_\e \in S(\Delta), (\Psi_{1,\e})_\e, \dotsc, (\Psi_{k,\e})_\e \in S^0(\Delta))$ $(\exists N \in \bN)$:
\begin{equation}\label{estimate}
 \norm { \ud_1^k \ud_2^l ( \widearrow R_{\sigma\tau} \circ (\id \times \varphi^{-1} ) ) (\Phi_\e, \varphi(p))(\Psi_{1,\e}, \dotsc, \Psi_{k,\e}) } = O(\e^{-N})
\end{equation}
uniformly for $p \in L \cap \underline{R}_\e^{-1}(L')$. 
\end{enumerate}
The space of all moderate elements of $\GHOMb{E}{F}$ is denoted by $\GHOMm{E}{F}$.
\end{definition}

\begin{definition}\label{homequiv}
Two elements $R, S \in \Hom^\Delta_\cE(E,F)$ are called vector bundle equivalent (vb-equivalent), written $R \sim_{vb} S$, if
\begin{enumerate}[label=(\roman*)]
\item \label{homequiv.1} $\underline{R} \sim \underline{S}$ in $\GMFm{E}{F}$,
\item \label{homequiv.2} $(\forall (U, \tau) \in \fS_E)$ $(\forall (U_0, \varphi) \in \cA_M)$ $(\forall (V, \sigma) \in \fS_F)$ $(\forall k,l \in \bN_0)$ $(\forall L \csub U_0 \cap U)$ $(\forall L' \csub V)$ $(\forall (\Phi_\e)_\e \in S(\Delta), (\Psi_{1,\e})_\e, \dotsc, (\Psi_{k,\e})_\e \in S^0(\Delta))$ $(\forall m \in \bN)$:
\begin{multline}\label{vierstrich}
 \bigl\lVert \ud_1^k \ud_2^l ( \widearrow R_{\sigma\tau} \circ (\id \times \varphi^{-1}))(\Phi_\e, \varphi(p))(\Psi_{1,\e}, \dotsc, \Psi_{k,\e})  \\
  - \ud_1^k \ud_2^l ( \widearrow S_{\sigma\tau} \circ (\id \times \varphi^{-1}))(\Phi_\e, \varphi(p))(\Psi_{1,\e}, \dotsc, \Psi_{k,\e}) \bigr\rVert = O(\e^m)
\end{multline}
uniformly for all $p \in L \cap \underline{R}_\e^{-1}(L') \cap \underline{S}_\e^{-1}(L')$.
\end{enumerate}
If \ref{homequiv.2} holds only for $k=l=0$, we say that $R,S$ are vector bundle equivalent of order 0 (vb0-equivalent), written $R \sim_{vb0} S$.
\end{definition}

\begin{remark}Definitions \ref{hommod} and \ref{homequiv} are independent of the chosen atlases and families of trivializing coverings. Moreover, vb-equivalence and vb0-equivalence define an equivalence relation on $\Hom^\Delta_\cE(E,F)$.
\end{remark}

\begin{definition}
 $\GHOMq{E}{F} \coleq \GHOMm{E}{F} / {\sim_{vb}}$ is the space of \emph{generalized vector bundle homomorphisms} from $E$ to $F$.
\end{definition}

Let us consider the case where $E = M \times \bE$ and $F = N \times \bF$ are trivial. Then the vector part $\widearrow R = \pr_2 \circ R$ of $R \in \GHOMb{M}{N}$ can be seen as a mapping $\widearrow R \in C^\infty ( \VSO{\Delta} \times M, \Lin ( \bE; \bF ))$ which has components $R^{ij}$ ($i=1 \dotsc \dim \bF$, $j = 1 \dotsc \dim \bE$) in $C^\infty(\VSO{\Delta} \times M, \bK) \cong C^\infty(\VSO{\Delta}, C^\infty(M)) = \cE^\Delta(M)$. From the definitions above, the following is evident:

\begin{lemma}\label{koord}
If $E$ and $F$ are trivial, $R \in \Hom^\Delta_\cE(E,F)$ can be written as
 \[ R = (\underline{R}, (\widearrow R^{ij})_{ij} ) \]
 with base map $\underline{R} \in \GMFb{M}{N}$ and components $\widearrow R^{ij} \in \cE^\Delta(M)$. $R$ is moderate if and only if $\underline R$ and all $\widearrow R^{ij}$ are moderate; $R \sim_{vb} S$ if and only if $\underline{R} \sim \underline{S}$ and $\widearrow R^{ij} - \widearrow S^{ij} \in \cN^\Delta(M)$ for all $i,j$. In particular, $R \sim_{vb} S$ if and only if $R \sim_{vb0} S$.
\end{lemma}

\begin{definition}Given $R \in \GHOMb{E}{F}$ and $S \in \GHOMb{F}{G}$ we define their composition $S \circ R \in \GHOMb{E}{G}$ by $(S \circ R)(\Phi,p) \coleq S (\Phi, R(\Phi,p))$ ($\Phi \in \VSO{\Delta}$, $p \in E$).
\end{definition}

In order to show that vb0-equivalence is the same as vb-equivalence we employ the fact that for every vector bundle we can find another one such that their direct sum is trivial (\cite[Theorem I, Section 2.5, p.~76]{GHV}). For trivial vector bundles one can then essentially argue coordinatewise and use the corresponding result for generalized functions.

\begin{lemma}\label{smoothcomp}Let $R \in \GHOMm{E}{F}$, $f \in \Hom(E_0, E)$ and $g \in \Hom(F, F_1)$. Then $g \circ R \circ f \in \GHOMm{E_0}{F_1}$. $R \sim_{vb0} R'$ implies $g \circ R \circ f \sim_{vb0} g \circ R' \circ f$, and $R \sim_{vb} R'$ implies $g \circ R \circ f \sim_{vb} g \circ R' \circ f$.
\end{lemma}
\begin{proof}
By Proposition \ref{fctcomp} we have the needed statements about the base map. For the vector part we simply note that for suitable local trivializations $\tau_0,\tau,\sigma,\sigma_1$,
\[ \widearrow{(f \circ R \circ g)}_{\sigma_0 \tau_1} (\Phi_\e,p) = \widearrow f_{\sigma_0 \sigma} ( \underline R(\Phi_\e, \underline{g}(p))) \cdot \widearrow R_{\sigma \tau}(\Phi_\e, g(p)) \cdot \widearrow g_{\tau\tau_1}(p). \]
Because derivatives of $\widearrow f_{\sigma_0\sigma}$ and $\widearrow g_{\tau\tau_1}$ are bounded on compact sets and $\underline{R}$ is c-bounded, the claim follows immediately.
\end{proof}

\begin{corollary}Given $R,S \in \GHOMm{E}{F}$ we have $R \sim_{vb0} S$ if and only if $R \sim_{vb} S$.
\end{corollary}
\begin{proof}
Choose $E'$ and $F'$ such that $E\oplus E'$ and $F \oplus F'$ are trivial; denote by $i_F$ and $p_E$ the canonical injection $F \to F \oplus F'$ and projection $E \oplus E' \to E$, respectively. By Lemmas \ref{koord} and \ref{smoothcomp} we have
\begin{align*}
 R \sim_{vb_0} S &\Longrightarrow \iota_F \circ R \circ p_E \sim_{vb0} \iota_F \circ S \circ p_E \\
& \Longleftrightarrow \iota_F \circ R \circ p_E \sim_{vb} \iota_F \circ S \circ p_E \\
& \Longrightarrow R \sim_{vb} S.\qedhere
\end{align*}
\end{proof}

\begin{proposition}\label{homcomp}Let $R \in \GHOMm{E}{F}$ and $S \in \GHOMm{F}{G}$. Then
\begin{enumerate}[label=(\roman*)]
 \item\label{homcomp.1} $S \circ R \in \GHOMm{E}{G}$,
 \item\label{homcomp.2} if $R' \in \GHOMm{E}{F}$ and $S' \in \GHOMm{F}{G}$ are such that $R \sim_{vb} R'$ and $S \sim_{vb} S'$, then $S \circ R \sim_{vb} S' \circ R'$.
\end{enumerate}
\end{proposition}
\begin{proof}
Choose vector bundles $E'$, $F'$ and $G'$ such that $E \oplus E'$, $F \oplus F'$ and $G \oplus F'$ are trivial.
\[
 \xymatrix{
E \ar[r]^R \ar@<1ex>[d] & F \ar[r]^S \ar@<1ex>[d] & G \ar@<1ex>[d] \\
E \oplus E' \ar@<1ex>[u] \ar[r]_{\widetilde R} & F \oplus F \ar@<1ex>[u] \ar[r]_{\widetilde S} & G \oplus G'
}
\]
Then the corresponding mappings $\widetilde R$ and $\widetilde S$ are moderate by Lemma \ref{smoothcomp} and moderateness of $R \circ S$ follows if we can show moderateness of $\widetilde S \circ \widetilde R$; for the latter we note that $\underline{\widetilde S} \circ \underline{\widetilde R}$ is moderate by Proposition \ref{fctcomp}. Moreover, because all involved bundles are trivial we can use Lemma \ref{koord} in order to see that the vector part of $\widetilde S \circ \widetilde R$ is given by componentwise multiplication, i.e., $\overrightarrow{(\widetilde S \circ \widetilde R)}^{ik} = \sum_j \overrightarrow{\widetilde R}^{ij} \overrightarrow{\widetilde S}^{jk}$. From this, \ref{homcomp.1} and \ref{homcomp.2} are clear in the general case.
\end{proof}

Hence, composition is well-defined as a map $\GHOMq{E}{F} \times \GHOMq{F}{G} \to \GHOMq{E}{G}$. Finally, we are in the position to define the tangent of a generalized mapping.

\begin{definition}
 For $R \in \cE^\Delta[M,N]$ the tangent mapping $\tang R \colon \VSO{\Delta} \times \tang M \to \tang N$ of $R$ is defined as $(\tang R)(\Phi,v) \coleq \tang (R(\Phi,.))\cdot v$.
\end{definition}

\begin{proposition}\label{tangmap}
 Let $R \in \cE^\Delta[M,N]$. Then
\begin{enumerate}[label=(\roman*)]
 \item\label{tangmap.1} $\tang R \in \GHOMb{\tang M}{\tang N}$ with $\underline{\tang R} = R$.
 \item\label{tangmap.2} If $R$ is moderate, $\tang R$ is so.
 \item\label{tangmap.3} If $R,S \in \cE^\Delta_\cM[M,N]$ satisfy $R \sim S$ then $\tang R \sim_{vb} \tang S$.
\end{enumerate}
\end{proposition}
\begin{proof}
 \ref{tangmap.1}: $\underline{\tang R} = R$ is clear. For smoothness of $TR$ fix a chart $(\tang V, \tang \psi)$ of $\tang N$ coming from a chart $(V, \psi)$ of $N$. Given $(\lambda_0, v_0) \in \VSO{\Delta}\times \tang M$, by smoothness of $R$ there exists a chart $(U,\varphi)$ on $M$ containing $\pi(v_0)$ and a $c^\infty$-open neighborhood $W$ of $\lambda_0$ such that $R(W \times U) \subseteq V$ and $\psi \circ R \circ (\id \times \varphi^{-1}) \in C^\infty(W \times \varphi(U), \psi(V))$. It follows that $\tang R(W \times \tang U) \subseteq \tang V$ and
\[ (\tang \psi \circ \tang R \circ (\id \times \tang  \varphi^{-1}))(\lambda,p,v) = \ud_2 ( \psi \circ R \circ (\id \times \varphi^{-1}))(\lambda,p)\cdot v \]
whence $\tang \psi \circ \tang R \circ (\id \times \tang \varphi^{-1}) \in C^\infty(W \times \varphi(U) \times \bR^m, \psi(V) \times \bR^n)$.

\ref{tangmap.2}: for moderateness of $\tang R$ fix charts $(U,\varphi)$ on $M$, $(V, \psi)$ on $N$ and corresponding vector bundle charts $(\tang U, \tang \varphi)$ and $(\tang V, \tang \psi)$. Let $k,l \in \bN_0$, $L \csub U$ and $L' \csub V$. Then for $\Phi_\e \in \VSO{\Delta}$ and $p \in L \cap \underline{R}_\e(L')$ we have (with $\tau,\sigma$ being the trivializations corresponding to $\tang \varphi, \tang \varphi$)
\[ (\widearrow R_{\sigma\tau} \circ (\id \times \varphi^{-1}))(\Phi_\e, \varphi(p)) = \ud_2 ( \psi \circ R \circ (\id \times \varphi^{-1}))(\Phi_\e, \varphi(p)) \]
which satisfies the moderateness estimates by assumption.

\ref{tangmap.3} is seen the same way.
\end{proof}

\section{Hybrid generalized functions}\label{sec6}

Spaces of hybrid generalized functions were introduced in \cite{genpseudo} in order to give meaning to the notaion of generalized sections along generalized mappings. This in turn is needed for the notion of a geodesic of a generalized pseudo-Riemannian metric.

\begin{definition}
 We set $\GHYBb{M}{F} \coleq C^\infty(\VSO{\Delta} \times M, F)$ and, for each of its elements $R$, $\underline{R} \coleq \pi_F \circ R \in \GMFb{M}{N}$.
\end{definition}

We use the notation $\widearrow R_\sigma(\lambda,p) \coleq \pr_2 \circ \sigma \circ R$ where it is defined for a local trivialization $(V, \sigma)$ of $F$, and set $R_\e \coleq R(\Phi_\e, .)$.

\begin{definition}\label{defhyb}
$R \in \GHYBb{M}{F}$ is called moderate if
\begin{enumerate}[label=(\roman*)]
 \item\label{defhyb.1} $\underline{R} \in \GMFm{M}{N}$,
 \item\label{defhyb.2} $(\forall (U,\varphi) \in \cA_M)$ $(\forall (V,\sigma) \in \fS_F)$ $(\forall L \csub U)$ $(\forall L' \csub V)$ $(\forall k,l \in \bN_0)$ $(\forall (\Phi_\e)_\e \in S(\Delta)$, $(\Psi_{1,\e})_\e$, $\dotsc$, $(\Psi_{k,\e})_\e \in S^0(\Delta))$ $(\exists N \in \bN)$:
\[ \norm { \ud_1^k \ud_2^l ( \widearrow R_{\sigma} \circ (\id \times \varphi^{-1}) )(\Phi_\e, \varphi(p)) ( \Psi_{1,\e}, \dotsc, \Psi_{k,\e} ) } = O(\e^{-N}) \]
uniformly on $p \in L \cap \underline{R}_\e^{-1}(L')$.
\end{enumerate}
The set of all such moderate $R$ is denoted by $\GHYBm{M}{F}$.
\end{definition}

\begin{definition}\label{defhyb_neg} $R,S \in \GHYBm{M}{F}$ are called equivalent (written $R \sim_h S$) if 
\begin{enumerate}[label=(\roman*)]
 \item $\underline{R} \sim \underline{S}$,
 \item \label{defhybsim.2} $(\forall (U,\varphi) \in \cA_M)$ $(\forall (V,\sigma) \in \fS_F)$ $(\forall L \csub U)$ $(\forall L' \csub V)$ $(\forall k,l \in \bN_0)$ $(\forall (\Phi_\e)_\e \in S(\Delta)$, $(\Psi_{1,\e})_\e$, $\dotsc$, $(\Psi_{k,\e})_\e \in S^0(\Delta))$ $(\forall m \in \bN)$:
\begin{multline*} \Bigl\lVert
 \ud_1^k \ud_2^l ( \widearrow R_{\sigma} \circ (\id \times \varphi^{-1}) )(\Phi_\e, \varphi(p)) ( \Psi_{1,\e}, \dotsc, \Psi_{k,\e} )  \\
- \ud_1^k \ud_2^l ( \widearrow S_{\sigma} \circ (\id \times \varphi^{-1}) )(\Phi_\e, \varphi(p)) ( \Psi_{1,\e}, \dotsc, \Psi_{k,\e} ) \Bigr\rVert = O(\e^m)
\end{multline*}
uniformly for $p \in L \cap \underline{R}_\e^{-1}(L') \cap \underline{S}_\e^{-1}(L')$.
\end{enumerate}
By $\sim_{h0}$ we denote the same relation but with $k=l=0$ in \ref{defhybsim.2}.
\end{definition}

\begin{definition}$\GHYBq{M}{F} \coleq \GHYBm{M}{F} / {\sim_h}$.
\end{definition}

Again, Definitions \ref{defhyb} and \ref{defhyb_neg} are independent of the atlas and trivializing coverings chosen and Definition \ref{defhyb_neg} gives equivalence relations. As in the previous section we first show that these definitions are functorial if we take as morphisms smooth maps and smooth vector bundle homomorphisms, respectively; this will allow us to reduce further questions to the case where $F$ is trivial.

\begin{definition}We define compositions $(R,S) \mapsto S \circ R$,
 \begin{align*}
  \GMFb{M}{N} \times \GHYBb{N}{G} &\to \GHYBb{M}{G} \\
  \GHYBb{M}{F} \times \GHOMb{F}{G} &\to \GHYBb{M}{G} 
 \end{align*}
by $(S \circ R)(\Phi,x) \coleq S(\Phi, R(\Phi,x))$.
\end{definition}
\begin{lemma}\label{lemmaB}
 Let $R \in \GHYBm{M}{F}$, $f \in C^\infty(M_0, M)$ and $g \in \Hom(F, F_1)$. Then $g \circ R \circ f \in \GHYBm{M_0}{F_1}$. $R \sim_h R'$ implies $g \circ R \circ f \sim_h g \circ R' \circ f$ and $R \sim_{h0} R'$ implies $g \circ R \circ f \sim_{h0} g \circ R' \circ f$.
\end{lemma}
\begin{proof}
 Proposition \ref{fctcomp} gives the claim about the base maps. For the vector part we note that for suitable local trivializations $\sigma,\sigma_1$ and charts $\varphi_0, \varphi$ we have
\begin{multline*}
( \overrightarrow{ ( g \circ R \circ f)}_{\sigma_1} \circ (\id \times \varphi_0^{-1}))(\Phi_\e, \varphi_0(p)) \\
= \widearrow g_{\sigma_1 \sigma} ( \underline{R} ( \Phi_\e, f(p))) \cdot ( \widearrow R_\sigma \circ (\id \times \varphi^{-1}))(\Phi_\e, ( \varphi \circ f \circ \varphi_0^{-1})(\varphi_0(p)))
\end{multline*}
from which the corresponding estimates can be obtained using the chain rule.
\end{proof}

The following is evident.

\begin{lemma}\label{lemmaA}
 If $F$ is trivial $R \in \GHYBb{M}{F}$ can be written as $R = (\underline{R}, (\widearrow R^i)_i)$ with base map $\underline{R} \in \GMFb{M}{N}$ and components $\widearrow R^i \in \cE^\Delta(M)$, $i=1 \dotsc \dim F$. In this case, $R$ is moderate if and only if $\underline{R}$ and all $\widearrow R^i$ are moderate; $R \sim_h S$ and $R \sim_{h0} S$ are both equivalent to $\underline{R} \sim \underline{S}$ and $\widearrow R^i - \widearrow S^i \in \cN^\Delta(M)$ for all $i$.
\end{lemma}

\begin{corollary}Given $R,S \in \GHYBb{M}{F}$, we have $R \sim_h S$ if and only if $R \sim_{h0} S$.
\end{corollary}
\begin{proof}Choose $F'$ such that $F \oplus F'$ is trivial. By Lemmas \ref{lemmaB} and \ref{lemmaA} we then have
\begin{align*}
 R \sim_{h0} S &\Longrightarrow i_F \circ R \sim_{h0} i_F \circ S \\
& \Longleftrightarrow i_F \circ R \sim_h i_F \circ S \\
& \Longrightarrow R \sim_h S.\qedhere
\end{align*}
\end{proof}
\begin{proposition}\label{hybcomp}
 \begin{enumerate}[label=(\roman*)]
  \item\label{hybcomp.1} For $R \in \GHYBb{M}{F}$ and $S \in \GHOMm{F}{G}$, $S \circ R \in \GHYBm{M}{G}$; $R \sim_h R'$ and $S \sim_{vb} S'$ imply $S \circ R \sim_h S' \circ R'$.
  \item\label{hybcomp.2} For $R \in \GMFm{M}{N}$ and $S \in \GHYBm{N}{G}$, $S \circ R \in \GHYBm{M}{G}$; if $R \sim R'$ and $S \sim_h S'$, then $S \circ R \sim_h S' \circ R'$.
 \end{enumerate}
\end{proposition}
\begin{proof}This is shown as in Proposition \ref{homcomp} by reduction to trivial bundles, such that the claims reduce to \ref{hybcomp.1} multiplication or \ref{hybcomp.2} composition of the vector components of the mappings involved.
\end{proof}

Hence, we have well-defined composition mappings
\begin{align*}
 \GMFq{M}{N} \times \GHYBq{N}{G} & \to \GHYBq{M}{G}, \\
 \GHYBq{M}{F} \times \GHOMq{F}{G} & \to \GHYBq{M}{G}.
\end{align*}

\section{Generalized sections along generalized mappings}\label{sec7}

\begin{lemma}\label{vektorraum}
\begin{enumerate}[label=(\roman*)]
 \item\label{vektorraum.1}  Let $R \in \GHOMm{E}{F}$ and $S \in \GMFm{M}{N}$ with $\underline{R} \sim S$. Then there exists $R_0 \in \GHOMm{E}{F}$ such that $R \sim_{vb} R_0$ and $\underline{R_0} = S$.
 \item\label{vektorraum.2} Let $R \in \GHYBm{M}{F}$ and $S \in \GMFm{M}{N}$ with $\underline{R} \sim S$. Then there exists $R_0 \in \GHYBm{M}{F}$ such that $R \sim_{vb} R_0$ and $\underline{R_0} = S$.
\end{enumerate}
\end{lemma}
\begin{proof}
\ref{vektorraum.1}: Supposing first that $E$ and $F$ are trivial we can set
\begin{equation}\label{donnerstag}
R_0 ( \Phi, p, v) \coleq (S(\Phi,p), \widearrow R( \Phi,p) \cdot v)
\end{equation}
with $\widearrow R (\Phi,p) \cdot v \coleq (\pr_2 \circ R)(\Phi,p,v)$. Obviously $\underline{R_0} = S$ and $R_0 \sim_{vb} S$ holds.

For the general case choose vector bundles $E'$ and $F'$ such that $E \oplus E'$ and $F \oplus F'$ are trivial. Define $P \in \GHOMm{E \oplus E'}{F \oplus F'}$ by $P \coleq \iota_F \circ R \circ p_E$ where $i_F$ and $p_E$ denote the canonical injection of $F$ and projection to $E$, respectively; we have $\underline{P} = \underline{R} \sim S$, hence by the first part we obtain $P_0 \in \GHOMm{E \oplus E'}{F \oplus F'}$ such that $P_0 \sim_{vb} P$ and $\underline{P_0} = S$. We then set $R_0 \coleq p_F \circ P_0 \circ i_E$. It follows that $R_0 \sim p_F \circ P \circ i_E = p_F \circ i_F \circ R \circ p_E \circ i_E = R$ and $\underline{R_0} = \underline{P_0} = S$.

For \ref{vektorraum.2} we replace \eqref{donnerstag} by $R_0(\Phi,p) \coleq (S(\Phi,p), \vec R(\Phi, p))$ and proceed similarly.
\end{proof}

\begin{corollary}For given $S \in \cG^\Delta[M,N]$ the sets
\begin{align*}
 \GHOMq{E}{F}(S) & \coleq \{\,R \in \GHOMq{E}{F}\ |\ \underline{R} = S \,\} \\
\GHYBq{M}{F}(S) & \coleq \{\,R \in \GHYBq{M}{F}\ |\ \underline{R} = S \,\}
\end{align*}
 are vector spaces.
\end{corollary}

We call $\GHYBq{M}{F}(S)$ the space of generalized sections along $S$. In particular, we denote by $\fX_\cG(S) \coleq \GHYBq{M}{TN}(S)$ the space of \emph{generalized vector fields on $S$}. This notion is necessary in order to define the notion of a geodesic of a generalized pseudo-Riemannian metric (cf.~\cite{genpseudo}).

\section{Point values}\label{sec8}

Although a definition of point values of distributions was introduced by S.~\L ojasiewicz in \cite{Lojasiewicz}, an arbitrary distribution need not have a point value at every point and it is not possible to characterize distributions by their point values.

An important feature of Colombeau algebras is the availability of such a point value characterization (\cite{ObeChar}). This concept was extended to the case of the special Colombeau algebra on manifolds in \cite{foundgeom} and to point values of generalized vector bundle homomorphisms in \cite{Kgenmf}. We will give the corresponding definitions in our setting and show the usual desired properties. Note in particular that a point value characterization in full diffeomorphism invariant Colombeau spaces of generalized functions is considerable more complicated to attain than in the special algebra because one has to construct suitable bump functions on spaces like $\cD(\bR^n)$, as is evidenced first in \cite{punktwerte}.

For applications of point values in Colombeau algebras we refer, e.g., to \cite{GroupAnalysis, DiscontCalculus, LocalProperties}. Furthermore, we point out that a point value characterization for full Colombeau algebras on manifolds is only available by our approach based on smoothing operators (or smoothing kernels) as parameters in the basic space; the more restrictive setting of \cite{global} is not able to accomodate point values in a coordinate-invariant way, as is seen by comparing this section with \cite{punktwerte}.

\begin{definition}\label{defgenvect}
 Let $\bE$ be a finite dimensional vector space over $\bK$ with norm $\norm{\cdot}$. We define
\begin{align*}
 \Gvb{\bE}{\Delta} &\coleq C^\infty(\VSO{\Delta}, \bE), \\
 \Gvm{\bE}{\Delta} &\coleq \{\, X \in \bE^\Delta : \forall k \in \bN_0\ \forall (\Phi_\e)_\e \in S(\Delta), (\Psi_{1,\e})_\e,\dotsc,(\Psi_{k,\e})_\e \in S^0(\Delta)\ \exists N \in \bN:\\
 & \qquad\qquad\qquad\qquad \norm{(\ud ^k R)(\Phi_\e)(\Psi_{1,\e}, \dotsc, \Psi_{k,\e})} = O(\e^{-N})\ \}, \\
 \Gvn{\bE}{\Delta} &\coleq
\{\, X \in \bE^\Delta: \forall k \in \bN_0\ \forall (\Phi_\e)_\e \in S(\Delta), (\Psi_{1,\e})_\e,\dotsc,(\Psi_{k,\e})_\e \in S^0(\Delta)\ \forall m \in \bN: \\
& \qquad\qquad\qquad\qquad \norm{(\ud ^k R)(\Phi_\e)(\Psi_{1,\e},\dotsc,\Psi_{k,\e})} = O(\e^m)\ \}, \\
 \Gvq{\bE}{\Delta}&\coleq \Gvm{\bE}{\Delta} / \Gvn{\bE}{\Delta}.
\end{align*}
Elements of $\Gvq{\bE}{\Delta}$ are called \emph{generalized $\bE$-vectors}, elements of $\Gvq{\bK}{\Delta}$ \emph{generalized numbers}. We abbreviate $R_\e \coleq R(\Phi_\e)$.
\end{definition}

$\Gvq{\bK}{\Delta}$ is a commutative ring and $\Gvq{\bE}{\Delta}$ a $\Gvq{\bK}{\Delta}$-module. The algebraic properties of Colombeau generalized numbers have been studied in various settings before, see for example \cite{zbMATH06264384,1241.46024,zbMATH05885461,zbMATH05768499}; it is expected that analogous results can be obtained in our setting also (barring the fact that smooth dependence on elements of $\VSO{\Delta}$ can lead to some technical complications).

The following is a variant of Theorem \ref{nonegder} for generalized vectors and can be proven analogously.

\begin{lemma}\label{nonegderpoints}
 If $X \in \Gvm{\bE}{\Delta}$ then $X \in \Gvn{\bE}{\Delta}$ if and only if $\forall m \in \bN$ $\forall (\Phi_\e)_\e \in \TOv{\Delta}$: $\abso{R_\e} = O(\e^m)$.
\end{lemma}

We will now define generalized points of manifolds and vector bundles. From now on we will use the notation $\abso{f_\e} = O(g_\e)$ also if $f_\e$ is only defined for $\e$ in a specified subset of $(0,1]$.

\begin{definition}\label{def_genpoint}
\begin{enumerate}[label=(\alph*)]
 \item Let $M^\Delta \coleq C^\infty(\VSO{\Delta}, M)$. By $M^\Delta_c$ we denote the set of all $X \in M^\Delta$ such that $(\exists L \csub M)$ $(\forall (\Phi_\e)_\e \in \TOv{\Delta})$ $(\exists \e_0>0)$ $(\forall \e < \e_0)$: $X_\e \in L$.

\item $X \in M_c^\Delta$ is called \emph{moderate} if $(\forall k \in \bN_0)$ $(\forall (U,\varphi) \in \cA_M)$ $(\forall L \csub U)$ $(\forall (\Phi_\e)_\e \in S(\Delta)$, $(\Psi_{1,\e})_\e$, $\dotsc$, $(\Psi_{k,\e})_\e \in S^0(\Delta))$ $(\exists N \in \bN)$:
\[ \norm{ \ud ^k (\varphi \circ X)(\Phi_\e)(\Psi_{1,\e}, \dotsc, \Psi_{k,\e})} = O(\e^{-N}) \]
for all $\e$ such that $X_\e \in L$. The set of moderate elements of $M^\Delta_c$ is denoted by $M^\Delta_{c,\cM}$.

\item \label{def_genpoint_neg}Given $X,Y \in M_{c,\cM}^\Delta$ we say that $X$ and $Y$ are equivalent, written $X \sim Y$, if
\begin{enumerate}[label=(\roman*)]
 \item \label{def_genpoint_neg.1} $(\forall h \in \Riem(M))$ $(\forall (\Phi_\e)_\e \in S(\Delta))$: $d_h ( X_\e, Y_\e) \to 0$,
 \item \label{def_genpoint_neg.2} $(\forall k \in \bN_0)$ $(\forall (U,\varphi) \in \cA_M)$ $(\forall L \csub U)$ $(\forall (\Phi_\e)_\e \in S(\Delta),(\Psi_{1,\e})_\e, \dotsc, (\Psi_{k,\e})_\e \in S^0(\Delta))$ $(\forall m \in \bN)$:
\[ \norm{ \ud ^k ( \varphi \circ X - \varphi \circ Y ) ( \Phi_\e) ( \Psi_{1,\e}, \dotsc, \Psi_{k,\e} ) } = O(\e^m) \]
for all $\e$ such that $X_\e \in L$ and $Y_\e \in L$.
\end{enumerate}
If \ref{def_genpoint_neg.2} holds only for $k=0$ we write $X \sim_0 Y$ (equivalence of order 0).
\item We set $\widetilde M^\Delta_c \coleq M^\Delta_{c,\cM} / {\sim}$ and call its elements \emph{compactly supported generalized points} of $M$.
\end{enumerate}
\end{definition}

Note that in Definition \ref{def_genpoint} \ref{def_genpoint_neg} \ref{def_genpoint_neg.1} one can replace $(\forall h \in \Riem(M))$ by $(\exists h \in \Riem(M))$. 
Again, because the $k$th differentials are symmetric we can take all $(\Psi_{1,\e})_\e \dotsc (\Psi_{k,\e})_\e$ to be equal. We have the following coordiante-free characterization of 0-equivalence:

\begin{proposition}\label{pnegcharm}
Let $X,Y \in M^\Delta_{c,\cM}$. Then the following are equivalent:
 \begin{enumerate}[label=(\alph*)]
  \item\label{pnegcharm.2} $X \sim_0 Y$;
  \item\label{pnegcharm.3} $(\forall h \in \Riem(M))$ $(\forall (\Phi_\e)_\e \in S(\Delta))$ $(\forall m \in \bN)$: $d_h(X_\e, Y_\e) = O(\e^m)$.
\end{enumerate}
\end{proposition}
\begin{proof}
This is seen exactly as in the proof of Theorem \ref{charnegcoordfree}.
\end{proof}

In the following, let $\pi_E \colon E \to M$ be a vector bundle with typical fiber $\bE$ (a finite dimensional vector space over $\bK$) and trivializing covering $\fS_E$. For $X \in C^\infty(\VSO{\Delta}, E)$ and $(U, \tau) \in \fS_E$, $\widearrow X_\tau \coleq \pr_2 \circ \tau \circ X$ denotes the vector part of $X$ with respect to $\tau$ where it is defined.

\begin{definition}
\begin{enumerate}[label=(\alph*)]
 \item Let $E^\Delta \coleq C^\infty(\VSO{\Delta}, E)$. For $X \in E^\Delta$ we set $\underline{X} \coleq \pi_E \circ X \in M^\Delta$.

\item We define the set $\GVBm{E}$ of \emph{moderate generalized vector bundle points} of $E$ as the set of all $X  \in E^\Delta$ such that
\begin{enumerate}[label=(\roman*)]
 \item $\underline{X} \in \GPm{M}$,
 \item $(\forall k \in \bN_0)$ $(\forall (U,\tau) \in \fS_E)$ $(\forall L \csub U)$ $(\forall (\Phi_\e)_\e \in \TOv{\Delta}, (\Psi_{1,\e})_\e \dotsc (\Psi_{k,\e})_\e \in \TOvz{\Delta})$ $(\exists N \in \bN)$:
\[ \norm { \ud ^k \widearrow X_\tau (\Phi_\e)(\Psi_{1,\e}, \dotsc, \Psi_{k,\e} ) } = O(\e^{-N}) \]
for all $\e$ such that $\underline{X}_\e \in L$.
\end{enumerate}
\item We say that $X, Y \in \GVBm{E}$ are \emph{vector bundle equivalent}, written $X \sim_{vb} Y$, if
\begin{enumerate}[label=(\roman*)]
 \item\label{vbpneg.1} $\underline{X} \sim \underline{Y}$,
 \item\label{vbpneg.2} $(\forall k \in \bN_0)$ $(\forall (\Phi_\e)_\e \in \TOv{\Delta}, (\Psi_{1,\e})_\e,\dotsc,(\Psi_{k,\e})_\e \in \TOvz{\Delta})$ $(\forall (U,\tau) \in \fS_E)$ $(\forall L \csub U)$ $(\forall m \in \bN)$:
\[ \norm { \ud^k (\widearrow X_\tau - \widearrow Y_\tau) ( \Phi_\e) ( \Psi_{1,\e}, \dotsc, \Psi_{k,\e} ) } = O(\e^m) \]
for all $\e$ such that $\underline{X}_\e, \underline{Y}_\e \in L$.
\end{enumerate}
If \ref{vbpneg.2} holds only for $k=0$ we write $X \sim_{vb0} Y$.
\item We set $\GVBq{E} \coleq \GVBm{E} / {\sim_{vb}}$ and call its elements \emph{generalized vector bundle points} of $E$.
\end{enumerate}
\end{definition}

\begin{lemma}If $E$ is trivial, $R \in \cE^\Delta(E)$ has components $R^i \in \cE^\Delta(M)$ ($i=1\dotsc \dim E$) and $R$ is moderate or negligible if and only if all $R^i$ are so.
 
Moreover, in this situation any generalized point $X \in E^\Delta = C^\infty(\VSO{\Delta}, M \times \bE)$ can be written as $X = (\pr_1 \circ X, \pr_2 \circ X) \eqcol (\underline{X}, \widearrow X)$ with $\widearrow X$ having components $\widearrow X^i \in \bK^\Delta$. $X$ is moderate if and only if $\underline{X}$ and $\widearrow X$ (or all $\widearrow X^i$) are so; $X \sim_{vb} Y$ if and only if $\underline{X} \sim \underline{Y}$ and $\widearrow X - \widearrow Y \in \bE^\Delta_\cN$ (or $\widearrow X^i - \widearrow Y^i \in \bK^\Delta_\cN$ for all $i$).
\end{lemma}

\begin{definition}We define point evaluation mappings
\begin{align*}
 \cE^\Delta(M) \times \GPb{M} & \to \bK^\Delta \\
 \cE^\Delta(E) \times \GPb{M} & \to \GVBb{E} \\
 \cE^\Delta[M,N] \times \GPb{M} &\to \GPb{N} \\
\intertext{by $R(X)(\Phi) \coleq R(\Phi)(X(\Phi))$ and }
 \cE^{\Delta,h}[M,F] \times \GPb{M} &\to \GVBb{F} \\
 \Hom^\Delta[E,F] \times \GVBb{E} &\to \GVBb{F}
\end{align*}
by $R(X)(\Phi) \coleq R(\Phi, X(\Phi))$. $R(X)$ is called the \emph{point value} of $R$ at $X$.
\end{definition}

In all of these cases, point evaluation respects moderateness and equivalence. We first show this for smooth mappings:
\begin{proposition}\label{smoothpointev}
 \begin{enumerate}[label=(\roman*)]
  \item\label{smoothpointev.1} Let $X \in \GPb{M}$ and $f \in C^\infty(M,N)$. If $X \in \GPc{M}$ then $f(X) \in \GPc{N}$; if $X \in \GPm{M}$ then $f(X) \in \GPm{N}$; For $X,Y \in \GPm{M}$ with $X \sim Y$, $f(X) \sim f(Y)$.
  \item\label{smoothpointev.2} Let $X \in \GVBb{E}$ and $f \in \Hom(E,F)$. If $X \in \GVBm{E}$ then $f(X) \in \GVBm{F}$. If $X,Y \in \GVBm{E}$ with $X \sim_{vb} Y$, then $f(X) \sim_{vb} f(Y)$.
 \end{enumerate}
\end{proposition}
\begin{proof}
\ref{smoothpointev.1} Preservation of c-boundedness is clear. Moderateness of $\ud^j ( f \circ X)(\Phi_\e)(\Psi_\e, \dotsc, \Psi_\e)$ is seen by applying the chain rule to $f \circ X = (f \circ \varphi^{-1}) \circ (\varphi \circ X)$ for a suitable chart $\varphi$, and similarly for equivalence.

\ref{smoothpointev.2} $\underline{f(X)} = \underline{f} ( \underline{X} )$ is moderate by \ref{smoothpointev.1}. Moderateness of $\ud^j \widearrow{ f(X) }_\sigma(\Phi_\e)(\Psi_\e, \dotsc, \Psi_\e)$ is seen by applying the chain rule to
$\widearrow { f(X) }_\sigma ( \Phi) = \widearrow f_{\sigma \tau}(\underline{X}(\Phi)) \cdot \widearrow X_\tau(\Phi)$ where $\tau,\sigma$ are suitable local trivializations of $E$ and $F$. Similarly, one sees $f(X) \sim_{vb} f(Y)$.
\end{proof}

Moreover, generalized points can be characterized by composition with smooth functions:
\begin{proposition}\label{chargenpt} Let $X \in \GPb{M}$. Then $X$ is moderate if and only if $f(X) \in \bK^\Delta$ is moderate for all $f \in C^\infty(M)$. For $X,Y \in \GPm{M}$, $X \sim Y$ if and only if $f(X) - f(Y) \in \bK^\Delta_\cN$ for all $f \in C^\infty(M)$, which hence is equivalent to $X \sim_0 Y$.
\end{proposition}
\begin{proof}
 This is obtained in exactly the same way as Theorems \ref{charmod} and \ref{charneg}.
\end{proof}

\begin{proposition}\label{genpoint_mod}
\begin{enumerate}[label=(\alph*)]
 \item\label{genpoint_mod.1} Given $R \in \Gm{M}{\Delta}$ and $X \in \GPm{M}$, $R(X) \in \Gvm{\bK}{\Delta}$. For $R \in \Gn{M}{\Delta}$ we have $R(X) \in \Gvn{\bK}{\Delta}$ and for $X \sim Y$ we have $R(X) - R(Y) \in \Gvn{\bK}{\Delta}$.
 \item \label{genpoint_mod.5} Given $R \in \Gm{E}{\Delta}$ and $X \in \GPm{M}$, $R(X) \in \GVBm{E}$. Furthermore, if $S \in \Gm{E}{\Delta}$ and $Y \in \GPm{M}$ satisfy $R - S \in \Gn{E}{\Delta}$ and $X \sim_0 Y$ then $R(X) \sim_{vb} S(Y)$.
\item\label{genpoint_mod.2} Given $R \in \GMFm{M}{N}$ and $X \in \GPm{M}$, $R(X) \in \GPm{N}$. Moreover, if $S \in \GMFm{M}{N}$ and $Y \in \GPm{M}$ are given with $R \sim S$ and $X \sim Y$, then $R(X) \sim_{vb} S(Y)$.
\item \label{genpoint_mod.4} Given $R \in \GHYBm{M}{F}$ and $X \in \GPb{M}$, $R(X) \in \GVBm{F}$. Moreover, if $S \in \GHYBm{M}{F}$ and $Y \in \GPm{M}$ are given with $R \sim_h S$ and $X \sim Y$, then $R(X) \sim_{vb} S(Y)$.
\item \label{genpoint_mod.3} Given $R \in \GHOMm{E}{F}$ and $X \in \GVBm{E}$, $R(X) \in \GVBm{F}$. Moreover, if $S \in \GHOMm{E}{F}$ and $Y \in \GVBm{E}$ satisfy $R \sim_{vb} Y$ and $X \sim_{vb} Y$, then $R(X) \sim_{vb} S(Y)$.
\end{enumerate}
\end{proposition}
\begin{proof}
\ref{genpoint_mod.1} We have to estimate
\begin{equation}\label{genpoint_mod.1eq1}
\ud^j ( R(X)) (\Phi_\e)(\Psi_\e, \dotsc, \Psi_\e).
\end{equation}
Because $X$ is c-bounded it suffices to obtain the estimate for all $\e$ with $X(\Phi_\e) \in L \csub U$ for a chart $(U,\varphi)$ on $M$. Applying the chain rule to
\[ R(X)(\Phi_\e) = (R(\Phi_\e) \circ \varphi^{-1})((\varphi \circ X)(\Phi_\e)) \]
we see that \eqref{genpoint_mod.1eq1} is given by terms of the form
\begin{multline*}
 \ud^l ( \ud^k R (\Phi_\e)(\Psi_\e, \dotsc, \Psi_\e) \circ \varphi^{-1} ) ((\varphi \circ X)(\Phi_\e)) \\
\cdot \ud^{i_1} (\varphi \circ X)(\Phi_\e, \dotsc, \Phi_\e) \cdot \dotsc \cdot \ud^{i_l} (\varphi \circ X)(\Phi_\e, \dotsc, \Phi_\e).
\end{multline*}
From this, the first two claims follow immediately. Now suppose $X \sim_0 Y$. Again, it suffices to obtain the estimate for all $\e$ with $X(\Phi_\e) \in L \csub U$ for a chart $(U,\varphi)$ on $M$. By Definition \ref{def_genpoint} \ref{def_genpoint_neg} \ref{def_genpoint_neg.1} we can choose a neighborhood $L' \csub U$ of $L$ such that $X(\Phi_\e) \in L$ implies $Y(\Phi_\e) \in L'$ for small $\e$. Moreover, we can without limitation of generality assume that $\varphi(L')$ is convex. In this case we can write $\abso{ R(X)(\Phi_\e) - R(Y)(\Phi_\e)}$ as
\begin{multline*}
 \abso{ \int_0^1 \frac{\pd}{\pd t}( R(\Phi_\e) \circ \varphi^{-1})((\varphi \circ Y)(\Phi_\e) + t \cdot (( \varphi \circ X)(\Phi_\e) - (\varphi \circ Y)(\Phi_\e)))\ud t} \\
\le \sup_{x \in \varphi(L')} \norm{ \ud(R(\Phi_\e) \circ \varphi^{-1})(x)} \cdot \norm{ (\varphi \circ X)(\Phi_\e) - (\varphi \circ Y)(\Phi_\e) }
\end{multline*}
from which we see that $R(X) - R(Y) \in \bK^\Delta_\cN$.

\ref{genpoint_mod.5} If $E$ is trivial, $R(X) = (X, (R^i(X))_i)$ gives the claim by \ref{genpoint_mod.1}. In the general case, choose $E'$ such that $E \oplus E'$ is trivial; then $R(X) = p_E ( i_E ( R(X)))$ and because $(i_E \circ R)(X)$ is moderate, $R(X)$ is moderate. For $R - S \in \cN^\Delta(E)$, $R(X) = p_E ( ( i_E \circ R)(X)) \sim p_E ( (i_E \circ S)(X)) = S(X)$ because $i_E \circ R - i_E \circ S \in \cN^\Delta(E \oplus E')$. Finally, $X \sim Y$ implies $R(X) = p_E((i_E \circ R)(X)) \sim p_E ( (i_E \circ R)(Y)) = R(Y)$.

\ref{genpoint_mod.2} $R(X)$ clearly is c-bounded. By Proposition \ref{chargenpt} it is moderate if and only if $f(R(X)) = (f \circ R)(X)$ is moderate for any $f \in C^\infty(N)$, which is the case by \ref{genpoint_mod.1} and Theorem \ref{charmod}, and similarly for equivalence.

\ref{genpoint_mod.4} If $F$ is trivial, $R(X) = (\underline{R}(X), (\widearrow R^i(X))_i)$ is moderate by \ref{genpoint_mod.1} and Lemma \ref{lemmaA}; similarly for equivalence. For non-trivial $F$ choose $F'$ such that $F \oplus F'$ is trivial; as in \ref{genpoint_mod.5}, the claim then follows by writing $R(X) = p_F ( (i_F \circ R)(X))$.

\ref{genpoint_mod.3} If $E$ and $F$ are trivial, the claim follows from $R(X) = ( \underline{R}(\underline{X}), (\sum_j \widearrow R^{ij} X^j)_i)$. In the general case, with $E \oplus E'$ and $F \oplus F'$ trivial, use $R(X) = p_F ( (i_F \circ R \circ p_E)(i_E (X)))$. 
\end{proof}

Consequently, point evaluation is well-defined on the quotients:

\begin{corollary}
The following point evaluation mappings are well-defined:
\begin{align*}
 \cG^\Delta(M) \times \GPq{M} &\to \Gvq{\bK}{\Delta} \\
 \cG^\Delta(E) \times \GPq{M} &\to \GVBq{E} \\
 \GMFq{M}{N} \times \GPq{M} &\to \GPq{N} \\
 \GHYBq{M}{F} \times \GPq{M} &\to \GVBq{F} \\
 \GHOMq{E}{F} \times \GVBq{E} &\to \GVBq{F}
\end{align*}
\end{corollary}

\section{Point value characterizations}\label{sec9}

In order to characterize a generalized function by its point values, the following auxiliary result (adapted from \cite{punktwerte}) allows to construct generalized points with specific properties. However, we need to establish some terminology for its formulation. First, in the local case (where $U \subseteq \bR^n$ is an open set) the topological vector space isomorphism $\cL ( \cD'(U); C^\infty(U)) \cong C^\infty(U, \cD(U))$ establishes a correspondence between smoothing operators $\Phi$ and smoothing kernels $\vec\varphi$; the conditions (VSO1-4,2',4') transform accordingly (cf.~\cite{bigone} for details), which gives an equivalent formulation of test objects in terms of smoothing kernels. The construction in the next lemma gives a generalized point defined in the local setting which can be pulled back to the manifold along a chart; there, the point of the corresponding submanifold of $M$ can be seen as a point on $M$ by pulling it back along the ``restriction mapping'' $\rSO_{U,M}$ of smoothing operators (see \cite[Theorem 5.6, p.~198]{bigone}), which maps test objects on $M$ to test objects on $U$.

For an open subset $U \subseteq \bR^n$, $S(U)$ and $S^0(U)$ denote the spaces of nets $(\vec\varphi_\e)_\e \in \SK{U}^I$ of smoothing kernels such that the corresponding nets of smoothing operators satisfy (VSO1)--(VSO4) or (VSO1,2',3,4'), respectively, of Section \ref{sec3} with $E = U \times \bK$ (cf.\ also \cite{papernew}). These are exactly the test objects used in the local scalar theory.

\begin{lemma}\label{localpoint}
Let $U \subseteq \bR^n$ be open, $(\vec\varphi_\e)_\e \in SK(U)^I$, $K \subseteq U$ open with $\overline{K} \subseteq U$ compact, and suppose we are given sequences $(\e_k)_k$ in $I$ and $(x_k)_k$ in $K$ such that $(\e_k)_k \to 0$. Then there exists $X \in C^\infty( SK (U), U)$ such that $X (\vec\varphi_{\e_k}) = x_k$ for infinitely many $k$ and $X$ is moderate in the sense that $(\forall j \in \bN_0)$ $(\forall (\vec\varphi_\e)_\e \in S(U))$ $(\forall (\vec\psi_{1,\e})_\e, \dotsc, (\vec\psi_{j,\e})_\e \in S^0(U))$ $(\exists N \in \bN)$: $\norm{(\ud^j X)(\vec\varphi_\e)(\vec\psi_{1,\e},\dotsc,\vec\psi_{j,\e})} = O(\e^{-N})$.
\end{lemma}
\begin{proof}
Let $z_0 \in U$ and $\eta>0$ arbitrary. Choose $r>0$ such that $\overline{B_r}(z_0) \subseteq U$. By choosing suitable subsequences we can assume that $(\e_k)_k$ is strictly decreasing, the convex hull of the set $\{ x_k: k \in \bN \}$ is contained in $U$, $\supp \vec\varphi_{\e_k}(z_0) \subseteq B_{r/2}(z_0)$ and $\norm{\vec\varphi_{\e_{k+1}}(z_0) }_\infty > \norm{\vec\varphi_{\e_k}}_{\infty} + 2 \eta$ for all $k \in \bN$, where $\norm{\cdot}_{\infty}$ denotes the supremum norm.

With $L \coleq \overline{B_r}(z_0)$, $\cD_L(U) \coleq \{ \varphi \in \cD(U):\supp \varphi \subseteq L\}$ is nuclear (\cite[Chapter III, \S 8.1, p.~106]{Schaefer}) and $W \coleq \{ \varphi \in \cD_L(U)\ |\ \norm{\varphi}_\infty \le \eta \}$ is a $0$-neighborhood in this space. By \cite[Chapter III, \S 7.3, p.~102]{Schaefer} there is a $0$-neighborhood $V \subseteq W$ whose gauge function is given by $p(\varphi) \coleq \sqrt{ \sigma(\varphi,\varphi) }$, where $\sigma$ is a positive semi-definite Hermitian form on $\cD_L(U)$. By the Cauchy-Schwartz inequality, $\abso{\sigma(\varphi,\psi)} \le p(\varphi)p(\psi)$ and hence $\sigma$ as well as its corresponding quadratic form $h(\varphi) \coleq \sigma (\varphi,\varphi)$ are smooth. 

Choose any $g \in C^\infty(\bR, [0,1])$ such that $g(x) = 1$ for $x \le 0$ and $g(x)=0$ for $x \ge 1$ and set $\chi \coleq g \circ h \in C^\infty(\cD_L(U), \bR)$. Take $\lambda \in C^\infty(U)$ with $\supp \lambda \subseteq L$ and $\lambda \equiv 1$ on $B_{r/2}(z_0)$. Then define $X \in C^\infty(\SK{U}, U)$ by
\[ X(\vec\varphi) \coleq x_1 + \sum_k \chi ( \lambda \cdot \vec\varphi(z_0) - \vec\varphi_{\e_k}(z_0)) \cdot (x_k - x_1). \]
Because $\chi^{-1}(0) \subseteq W$ the summands have pairwise disjoint carriers and $X(\vec\varphi_{\e_k}) = x_k$ for all $k$. For moderateness we have to show that $\forall j \in \bN_0$ $\forall (\vec\varphi_\e)_\e \in S(U)$, $(\vec\psi_\e)_\e \in S^0(U)$ $\exists N \in \bN$: $\norm{(\ud^j X)(\vec\varphi_\e)(\vec\psi_\e, \dotsc, \vec\psi_\e)} = O(\e^{-N})$. For each $\e \in I$ at most one term $\chi(\lambda \cdot \vec\varphi_\e(z_0) - \vec\varphi_{\e_k}(z_0))$ is nonzero. In this case, we know that $h(\lambda \cdot \vec\varphi_\e(z_0) - \vec\varphi_{\e_k}(z_0)) < 1$ and thus $\norm{\lambda \cdot \vec\varphi_\e(z_0) - \vec\varphi_{\e_k}(z_0)}_\infty < \eta$, hence $\norm{\lambda \cdot (\vec\varphi_\e + t \vec\psi_\e) - \vec\varphi_{\e_k}(z_0)}_\infty < \eta$ for $t$ sufficiently small. From this it follows that
\begin{multline*}
 (\ud^j X)(\vec\varphi_\e)(\vec\psi_\e, \dotsc, \vec\psi_\e) = \left.\left(\frac{\pd}{\pd t}\right)^j\right|_{t=0} X(\vec\varphi_\e + t\vec\psi_\e) \\
= \left.\left(\frac{\pd}{\pd t}\right)^j\right|_{t=0} \chi(\lambda \cdot (\vec\varphi_\e + t \vec\psi_\e)(z_0) - \vec\varphi_{\e_k}(z_0))(x_k-x_1) \\
= \ud^j [ \vec\varphi \mapsto (g \circ h)(\lambda \cdot \vec\varphi(z_0) - \vec\varphi_{\e_k}(z_0)) \cdot (x_k-x_1)](\vec\varphi_\e)(\vec\psi_\e, \dotsc, \vec\psi_\e).
\end{multline*}
Applying the chain rule we obtain derivatives of $g$, which are bounded independently of $\e$, and terms of the form
\begin{equation}\label{asfldfh}
 \ud^l h ( \lambda \cdot \vec\varphi(z_0) - \vec\varphi_{\e_k}(z_0)) (\vec\varphi_\e(z_0)) (\lambda \cdot \vec \psi_\e(z_0), \dotsc, \lambda \cdot \vec\psi_\e(z_0)).
\end{equation}
For $l \ge 3$ this vanishes because $h$ is a polynomial of degree 2. For the other cases we have $h(\varphi) = \sigma(\varphi, \varphi)$, $\ud h (\varphi)(\psi) = \sigma (\varphi, \psi) + \sigma( \psi, \varphi)$ and $\ud^2 h(\varphi)(\psi_1, \psi_2) = \sigma(\psi_1,\psi_2) + \sigma(\psi_2, \psi_1)$. By the Cauchy-Schwartz inequality, \eqref{asfldfh} hence can be estimated by terms which are either of the form $p(\lambda \cdot \vec\varphi(z_0) - \vec\varphi_{\e_k}(z_0))$, which is $<1$ as noted above, or of the form $p(\lambda \cdot \vec\psi_\e(z_0))$, which is $O(\e^{-N})$ for some $N \in \bN$ because $\vec\psi \mapsto p(\lambda \cdot \vec\psi(z_0))$ is a continuous seminorm on $\SK{U}$ and $(\vec\psi_\e)_\e \in S^0(U)$. Alltogether, this gives moderateness of $X$.
\end{proof}

\begin{theorem}\label{pvchar}
\begin{enumerate}[label=(\alph*)]

\item\label{pvchar.1} If $R \in \cG^\Delta(M)$, $R = 0$ if and only if $R(X) = 0$ for all $X \in \GPq{M}$.
\item\label{pvchar.2} If $R,S \in \cG^\Delta(E)$, $R=S$ if and only if $R(X)=S(X)$ for all $X \in \GPq{M}$.
\item\label{pvchar.3} If $R,S \in \GMFq{M}{N}$, $R=S$ if and only if $R(X) = S(X)$ for all $X \in \GPq{M}$.
\item\label{pvchar.4} If $R,S \in \GHYBq{M}{F}$, $R=S$ if and only if $R(X) = S(X)$ for all $X \in \GPq{M}$.
\item\label{pvchar.5} If $R,S \in \GHOMq{E}{F}$, $R=S$ if and only if $R(X) = S(X)$ for all $X \in \GVBq{E}$.
\end{enumerate}
\end{theorem}
\begin{proof}
\ref{pvchar.1} If $R=0$ then $R(X)=$ for all $X \in \widetilde M_c^\Delta$ by Proposition \ref{genpoint_mod}. Conversely, for each $R \in \Gm{M}{\Delta} \setminus \Gn{M}{\Delta}$ we can find a compact subset $K \csub M$, $m \in \bN_0$, $(\Phi_\e)_\e \in S(\Delta)$, a sequence $(\e_k)_k$ in $I$ with $\e_k < 1/k$ and a sequence $(x_k)_k$ in $K$ such that $\abso{R(\Phi_{\e_k})(x_k)} > \e_k^m$.

Pick any $G \in \Delta$ and, by choosing a suitable subsequence of $(x_k)_k$, choose a chart $(U, \varphi)$ of $M$ such that $G$ is trivial over $U$ and we can assume $K \subseteq U$. With respect to a fixed basis of $\Gamma(U,G)$, $\rSO_{U,M} \Phi_{G, \e}$ is given by a matrix $((\rSO_{U,M} \Phi_{G, \e})^{ij})_{ij}$ with entries in $\SO{U}$ (cf.~\cite[Section 5]{bigone} for the definition of $\rSO_{U,M}$ and the coordinate representation of test objects).
By Lemma \ref{localpoint} there exists a moderate generalized point $\widetilde X \in C^\infty(\SO{U}, U)$ such that $\widetilde X( (\rSO_{U,M} \Phi_{G,\e_k})^{11} ) = x_k$ for infinitely many $k$. Define $X \in M^\Delta_c$ by $X(\Phi) \coleq \widetilde X ( (\rSO_{U,M} \Phi_G )^{11} )$ for $\Phi \in \VSO{\Delta}$. Then $X$ is moderate and $X(\Phi_{\e_k} ) = x_k$ for infinitely many $k$ by construction. Because $R(X)(\Phi_{\e_k}) = R(\Phi_{\e_k})(x_k)$ for these $k$ it follows from the assumption that $R(X)$ cannot be negligible.

\ref{pvchar.2} If $E$ is trivial, $R(X) = (X, ( \overrightarrow{ R(X) }^i)_i)$ and $S(X) = (X, ( \overrightarrow{ S(X) }^i)_i)$ are equal if and only if all $\overrightarrow{ R(X) }^i = R^i(X)$ and $\overrightarrow{ S(X) }^i = S^i(X)$ are equal, which gives the claim by \ref{pvchar.1}. For nontrivial $E$ choose $E'$ such that $E \oplus E'$ is trivial; then $R(X) = S(X)$ for all $X \in \GVBq{E}$ is equivalent to $(i_E R)(X,Y) = (i_E S)(X,Y)$ for all $(X,Y) \in \GVBq{E \oplus E'}$, which in turn is equivalent to $i_E R = i_E S$ by \ref{pvchar.1} and hence $R=S$.

\ref{pvchar.3} $R(X) = S(X)$ for all $X \in \GPq{M}$ is equivalent to $f(R(X)) = f(S(X))$ for all $X \in \GPq{M}$ and $f \in C^\infty(N)$, which by Proposition \ref{chargenpt} is equivalent to $f \circ R = f \circ S$ for all $f \in C^\infty(N)$ and hence to $R=S$ by Theorem \ref{charneg}.

\ref{pvchar.4} Supposing first that $F$ is trivial, $R(X) = ( \underline{R}(X), (\widearrow R^i(X))_i)$ and $S(X) = ( \underline{S}(X), (\widearrow S^i(X))_i)$ are equal if and only if $\underline{R}(X) = \underline{S}(X)$ and $\widearrow R^i(X) = \widearrow S^i(X)$ for all $i$, which implies the claim. For general $F$, choose $F'$ such that $F \oplus F'$ is trivial. Then $R(X) = S(X)$ $\forall X$ implies $(i_F \circ R)(X) = (i_F \circ S)(x)$ $\forall X$, hence $i_F \circ R = i_F \circ S$ and thus $R = S$.

\ref{pvchar.5} If $E,F$ are trivial then $R(X)=S(X)$ if and only if $(\underline{R} ( \underline{X} ), (\widearrow {R(X)}^i)_i )$ and $(\underline{R} ( \underline{X} ), ( \widearrow { S(X)}^i )_i)$ are equal, which is the case if we have $\underline{R} ( \underline{X} ) = \underline{S} ( \underline{X} )$ as well as equality of $ \widearrow { R(X) }^i = \sum_j \widearrow R^{ij} X^j$ and $\widearrow { S(X) }^i = \sum_j \widearrow S^{ij} X^j$. If this is the case for all $X$, then $\underline R = \underline S$ and $\widearrow R^{ij} = \widearrow S^{ij}$ for all $i,j$ by \ref{pvchar.1} and \ref{pvchar.3}, which implies $R=S$. For arbitary $E,F$ choose $E'$ and $F'$ such that $E \oplus E'$ and $F \oplus F'$ are trivial. Then $R(X)=S(X)$ for all $X$ is equivalent to $(i_F \circ R \circ p_E)(X,Y) = (i_F \circ S \circ p_E)(X,Y)$ for all $(X,Y) \in \GPq{E \oplus E'}$, which in turn is equivalent to $i_F \circ R \circ p_E = i_F \circ S \circ p_E$ and hence $R = S$.
\end{proof}

{\bfseries Acknowledgments.} This work was supported by the Austrian Science Fund (FWF) projects P23714 and P26859.

\end{document}